\documentclass{article}
\usepackage{graphicx} 
\usepackage[utf8]{inputenc}
\usepackage{enumitem}
\usepackage{varwidth}
\usepackage{soul}
\usepackage{cancel}
\usepackage{tasks}
\usepackage[italicdiff]{physics}
\usepackage{amsmath}
\usepackage{amssymb}
\usepackage{algpseudocode}
\usepackage[linesnumbered,ruled,vlined]{algorithm2e}
\usepackage{enumitem}
\usepackage{subcaption}
\usepackage[top=2.5cm, bottom=2.5cm, left=4cm, right=4cm]{geometry}
\usepackage{graphicx}
\usepackage{amsthm} 
\newtheorem{theorem}{Theorem}
\newtheorem{prop}[theorem]{Proposition}

\newtheorem{remark}[theorem]{Remark}

\newenvironment{notation}
{\par\noindent\textbf{Notations:}\quad}  
  {\par}   
\newtheorem{assumption}{Assumption}

\usepackage{graphicx}

\usepackage[colorinlistoftodos]{todonotes}

\usepackage[ruled,vlined]{algorithm2e}

\newtheorem{lemma}[theorem]{Lemma}

\providecommand{\keywords}[1]
{
  \small	
  \textbf{\textit{Keywords---}} #1
}
\providecommand{\classification}[1]
{
  \small	
  \textbf{{AMS subject classifications--}} #1
}

\title{Adaptive Iterative Numerical Homogenization  for Quasilinear Nonmonotone Elliptic PDE}\author{Maher Khrais\footnotemark[1] \and
 Barbara Verfürth\footnotemark[1]}  
\date{}

\begin{document}

\maketitle
\renewcommand{\thefootnote}{\fnsymbol{footnote}}
\footnotetext[1]{Institut für Numerische Simulation, Universität Bonn, Friedrich-Hirzebruch-Allee 7, D-53115 Bonn, Germany}
\renewcommand{\thefootnote}{\arabic{footnote}}
 \begin{abstract}
    We propose and analyze an adaptive iterative numerical homogenization  method to approximate the solution of a class of quasilinear nonmonotone elliptic problems  that is of multiscale nature. The method  is based on  the technique of the Localized Orthogonal Decomposition (LOD) applied to the linear problems in each step of a Kačanov iteration.  In this approach,  the multiscale basis is   recomputed adaptively in each iteration and a linear problem is solved with this updated multiscale space, where in both steps the nonlinearity is evaluated using the approximation of the previous iteration.  As a key component of the proposed approach, we present a locally computable error indicator,  which at each iteration identifies the basis functions requiring updates, while previously computed basis functions are retained whenever possible. We provide a priori error estimates and show convergence of the method, requiring only higher integrability of the right-hand side, but no higher differentiability of the solution itself. Furthermore,  we   discuss how to adapt the proposed  adaptive iterative LOD in the context of Newton's method as iteration scheme. Numerical experiments illustrate the theory and validate the applicability of the proposed method.
\end{abstract}

 \\

\classification{65N30, 65N15, 35J60, 65J15, 65N12 }

\section{Introduction}\label{SectionIntro}
Over the past decades, manufactured composite materials  with tailored properties have attracted substantial growing interest   due to their multifunctionality and  wide range of applications in several fields of sciences.  Among these applications are, for example,  composite textile,  soft  and  reinforced composed  biological material,  and carbon-fiber, see \cite{bilogocalmethod, softbio, carbonfiber}. Heterogeneous materials 
can exhibit  nonlinear deformation or other  physical behaviors that are nonlinear when subject to high temperatures, intensities,  or forces.   This nonlinear physical response of  heterogeneous materials  necessitates accurate nonlinear multiscale modeling \cite{material}.
    In this work, we present, analyze, and implement an iterative numerical homogenization technique applied to a class of multiscale nonmonotone quasilinear elliptic PDEs of the following form
    \begin{equation}\label{eq:1}
        - \nabla \cdot(\alpha(x, u)\nabla{}u) = f \quad\text{in } \Omega. 
    \end{equation}
The domain  $\Omega \subset \mathbb{R}^d $ is an open bounded domain  $\text{with}\;d\leq 3$. 
Boundary conditions and further assumptions  will be made precise below.  Nonlinear PDEs of of the form  \eqref{eq:1} describe a range of physical processes,  such as  stationary Richards equation that
models the flow of the ground water in  unsaturated porous media,  or heat conductivity of  composite materials, see \cite{RicharEq, heatConduction}. 

Several authors have investigated  the well-posedness of problem \eqref{eq:1} in the context of weak solutions.    The uniqueness of the weak solution $u$ was established in \cite{Douglas}, and subsequently  derived as a weak limit of  Galerkin approximations in \cite{Unique}.  In addition, several numerical techniques have been considered to solve nonmonotone  problems such as \eqref{eq:1}. However, these studies mainly rely on higher regularity assumptions for the weak solution beyond $H^1_0(\Omega)$, or the assumption of a globally  fine mesh size; see, for example \cite{reg2, reg1, reg3}.  Nonetheless,   the uniqueness of the FEM solution in one and two dimensions without fine global mesh size assumption has been established in \cite{SaraPollock}.   We also refer to \cite{uniqandnotunique,DouglasGalerkin}  and the references therein for further arguments and details on the existence and uniqueness of  the approximate solutions.   In addition to difficulties associated with the nonlinearity, our problem involves a spatially heterogeneous coefficient.  Hence, to obtain a reliable approximation using classical numerical approaches, and to recover the optimal convergence rate w.r.t.\ the $H^1$-norm, the microstructural  features need to be resolved by the mesh size. In this case, numerical simulation would incur significant computational cost.  This has triggered growing interest in development of numerical homogenization methods.  

From a general perspective,  numerical homogenization techniques aim to solve a  global  macroscopic problem using modified multiscale basis functions that incorporate  the underlying fine-scale behaviors, and  that are obtained through microscopic solvers on local subdomains. A broad range of  numerical multiscale methods have been developed recently to address linear multiscale problems, in particular linear elliptic PDEs. These methods  include, for example, the heterogeneous multiscale method (HMM) \cite{HMM}, the multiscale finite element method (MsFEM) \cite{MsFEMLINEAR}, the generalized multiscale method finite element method (GMsFEM) \cite{GMSFEM}, the variational multiscale method (VMS) \cite{VMSM}, rough polyharmonic splines (RPS) \cite{RPS}, and the local orthogonal decomposition method (LOD) \cite{LOD-Linear}.  Certain numerical homogenization  techniques have been adapted to nonlinear multiscale problems,  with or without relying on a monotonicity assumption. The HMM has been formulated and analyzed for the nonmonotone nonlinear elliptic problem \eqref{eq:1}  in \cite{HMM_quasiliner, HMS-ofNonmonotne}. A component-based parametric model order reduction (CB-pMOR) technique is proposed for the parametrized nonmonotone   problem in~\cite{CB_pMOR}. The  MsFEM is considered for a monotone nonlinear PDE,     and the corresponding  convergence rate is  established in \cite{MSFEM}. Furthermore, the GMsFEM is implemented and analyzed to solve nonlinear elliptic problems in~\cite{NonlinearGMS,CEM-GMsFEM}. An iterative numerical homogenization in the context of generalized RPS with  monotone nonlinearity is considered in \cite{GRPS}. In such an iterative approach, the nonlinear problem is solved through a sequence of  approximate linear problems with updated coarse space. Using the so-called quasi-norm, the approach achieves an exponential energy convergence  that is independent of the heterogeneity. 

The iterative numerical homogenization considered in our work  is based on the LOD technique, which  was first presented and  successfully   applied to a range of (linear) problems, see \cite{ NumericalHomogenizationActaNumerica, LODFIRst} for overviews. Generally,  in the LOD method, we enrich  the coarse FE basis with locally computed  fine-scale corrections that resolve the details of the underlying microstructure,   resulting in a new modified low-dimensional function space with good approximation properties.  The fine-scale corrections are computed in a localized fashion on small patches of coarse mesh elements. Concerning nonlinear problems, the LOD has been investigated in the case of  monotone and nonmonotone quasilinear elliptic PDEs in \cite{ khraisverfuerth2025,Barbara}, and nonlinear Helmholtz problems \cite{MAir&barbara}.  In the approach of \cite{ khraisverfuerth2025,Barbara},  a multiscale space is constructed  only once and then the global nonlinear discrete problem is (iteratively) solved.   Thus, the construction  of the multiscale space depends on a suitable choice of linearization point.  The numerical experiments in \cite{ khraisverfuerth2025} reveal that, on the one hand, an inappropriate  linearization point could destroy the first-order convergence w.r.t.\ the mesh size, but that, on the other hand, using the first (inappropriate)  LOD  solution as a new linearization point, the first-order convergence is recovered. This motivates further investigation into the use of the iterative LOD as a potentially effective approach.   The idea is to iteratively update the  multiscale space in order to address the challenge of selecting a suitable linearization point.  
 
In this article,  we propose  an  adaptive iterative LOD technique for the nonlinear nonmonotone elliptic problem \eqref{eq:1}, which is inspired by the ideas presented for a nonlinear Helmholtz problem in \cite{MAir&barbara}.   In  the iterative LOD,  we  construct a new multiscale space and  solve a linear problem in each iteration based on a fixed point method using an approximation of the previous solution to ``freeze'' the nonlinearity. 
Different from \cite{khraisverfuerth2025, Barbara}, in the algorithms described in this work, both the correction and the  linear discrete problem are computed starting via the same initial point $\phi$ and the same  iterative method.  As the iterative LOD  typically requires a repeated number of multiscale space constructions entailing rather high computational cost, we apply an adaptive update strategy combining ideas from \cite{adaptive2, adaptive1} and \cite{MAir&barbara}. 
Furthermore, we  present a rigorous error analysis. 
  For sufficiently small tolerance for the adaptive updates, we show an a priori error estimate which is of optimal order in the mesh size without any dependence on the spatial variations of $\alpha$. In contrast to \cite{MAir&barbara}, the nonlinearity affects a higher-order term w.r.t.\ the derivatives, which requires substantially new arguments. A major challenge is that the weak solution of \eqref{eq:1} is of rather low regularity for heterogeneous coefficients. As a remedy, we assume higher integrability of the right-hand side $f$ and/or a limited contrast of the diffusion coefficient as suggested in \cite{LP}, resulting in a slightly improved regularity of the weak solution.  Via fixed-point arguments, we show convergence of the iterative LOD, requiring a certain smallness assumption on the data of the problem.
 
The article is structured as follows.  In  Section \ref{Sec: setting}, we introduce the model problem and present both the iterative Kačanov method and the finite element discretization. In Section \ref{Section :Ortho}, we introduce our iterative multiscale approach,  in which we construct a multiscale space and  solve a linear problem in each step. In Section \ref{Section: error1-2-3D}, the essential parts of our error analysis are presented, together with the corresponding regularity requirements. 
Further steps to make our approach practically feasible, such as the localization and adaptive update of the multiscale basis, are presented and analyzed in Section \ref{sec: adaptive}. In Section \ref{Sec:NEWTONAlgorithm},  we discuss how to apply the adaptive LOD in the framework of Newton iterations. Several numerical experiments are presented in Section \ref{sec: SectionExperiment} which demonstrate the effectiveness and applicability of the methods developed in the previous sections.  

\begin{notation}
    Throughout the article, we use standard notation on Lebesgue and Sobolev spaces and their norms. For a given subdomain $\Omega_1 \subseteq \Omega$, let $|\cdot|_{1,\Omega_1}$, $\|\cdot \|_{1,\Omega_1}$, $\|\cdot\|_{\infty}$,  $\|\cdot\|_{0,\Omega_1}$, and $\|\cdot\|_{L^p,\Omega_1}$ denote the $H^1(\Omega_1)$-semi-norm, the standard $H^1(\Omega_1)$-norm, the $L^\infty$-norm,  the standard $L^2(\Omega_1)$-norm, and the standard $L^p(\Omega_1)$-norm, respectively. The scalar product $(\cdot, \cdot)_{\Omega_1}$ is the $L^2$ inner product on the subdomain $\Omega_1$. We will omit the subscript $\Omega_1$ if it is the full domain, i.e., $\Omega_1=\Omega$. We write $a \lesssim b$ if there is a generic constant $C$ (independent from the discretization and multiscale parameters) such that $a\leq C b$. To simplify the presentation of some details and proofs,  we omit the spatial term associated with  the coefficient $\alpha(x, \cdot)$ or $\alpha_u(x, \cdot)$ whenever it is clear from the context. 
    \end{notation}

\section{Setting}\label{Sec: setting}
In this section, we introduce the model problem for our article and its main assumptions. To solve the nonlinear problem in practice, we also present an iterative scheme and  discuss briefly  the finite element method for it. In particular, we  underline the limitations of the FEM  in our multiscale context. 

\subsection{Problem formulation}

The model of interest is the following nonlinear nonmonotone PDE
\begin{equation}\label{PDEProblem}
       \begin{split}
           -\nabla\cdot{(\alpha(x,u)\nabla{u})}&=f \;\;\text{in}\;\;\Omega,  \\ u&=0 \;\; \text{on} \;\; \partial \Omega.
        \end{split}
\end{equation}
We consider a bounded Lipschitz domain $\Omega \subset \mathbb{R}^d,  d \in \{1,2,3\}$     with Lipschitz  boundary $\partial \Omega$. 
For now,  we focus our presentation on the case $ \;  d \in \{1,2\}$. Later, we will provide further explanations and outline the  assumptions for the $\operatorname{3D}$ domain in the relevant section.  Within our setting, we assume that $\alpha(\cdot,s) \in L^\infty (\Omega, \mathbb{R})$ is a strongly heterogeneous and highly varying coefficient with respect to the spatial argument $x$. More precisely, we make the following assumptions on the nonlinearity, which also guarantee the well-posedness of problem~\eqref{PDEProblem}.  
\begin{assumption}\label{assum1} Suppose that
\begin{itemize}
\item     $\alpha(x,s)$ is uniformly Lipschitz continuous w.r.t.\ the second argument, i.e., there exists  $L_c >0$ such that 
    \begin{equation}\label{LipCon}
       \begin{split}
       \abs{\alpha(x, s_1)-\alpha(x,s_2)}\leq L_c \abs{s_1-s_2},   \quad\forall x \in \Omega, \quad s_1, s_2 \in \mathbb{R}.
       \end{split}
    \end{equation}
  
 \item  $\alpha(x,s)$ satisfies the uniform ellipticity and boundedness conditions, i.e., there exist $\lambda >0$, and $\Lambda>0$ such that
    \begin{equation}\label{ellip}
                  {\alpha(x, s)\psi\cdot\psi} \geq \lambda \abs{\psi}, \quad \text{and} \quad |\alpha(x,s)\psi|\leq \Lambda _1 \abs{\psi} \quad \forall  x \in \Omega, \;s \in \mathbb{R}, \;  \psi \in \mathbb{R}^d.
    \end{equation}
    \end{itemize}
\end{assumption}
In addition,  we consider that $f \in L^2(\Omega)$. Later, we will assume  higher integrability on $f$ in the relevant sections. For $u, v \in H^1_0(\Omega)$ and fixed $z \in H^1_0(\Omega)$, we define  
\[\mathcal{A}(z;u,v) :=\int_\Omega\alpha(x,z)\nabla u\cdot\nabla v\, dx, \quad \text{ and } F(v):=\int_\Omega fv\, dx, \quad \forall  v  \in H^1_0(\Omega).\]
Consider the weak formulation associated with problem \eqref{PDEProblem}: find $u \in H_0^1(\Omega)$ that solves

\begin{equation}\label{WeakForm}
   \mathcal{A}(u;u,v)=F(v),  \quad \forall  v  \in H^1_0(\Omega).
\end{equation}
Provided  Assumption \ref{assum1} on the coefficient $\alpha$,   the existence and   uniqueness of the solution $u \in H^1_0(\Omega)$ of  problem \eqref{WeakForm} can be ensured; see,   for example \cite{ Douglas,Unique}.  The weak solution $u$ satisfies the following stability
\begin{equation*} 
    |u|_1 \leq C \|f\|_0,
\end{equation*} 
where $C$ is independent of the spatial variations of $\alpha$~\cite{chipot}, but depends on the constants of   assumption   \eqref{ellip}.  

Throughout our article, the following auxiliary linear problem will play an important role. For a fixed so-called linearization point $\phi\in H_0^1(\Omega)$, we seek $u_\phi \in H^1_0(\Omega)$ that solves
\begin{equation}\label{auxLinProb}
    \mathcal{A}(\phi;u_\phi,v)=F(v), \quad \forall  v \in H^1_0(\Omega).
\end{equation}
The bilinear form $\mathcal{A}(\phi; \cdot, \cdot)$  retains the ellipticity and boundedness conditions in $H^1_0(\Omega)$, i.e., there exists $\lambda>0 \;\text{and} \; \Lambda>0$ such that
\begin{align*}
    \lambda |w|_1^2 &\leq \mathcal{A}(\phi;w,w),  \quad \forall \phi,w \in H^1_0(\Omega), \\
    \mathcal{A}(\phi;v,w) &\leq \Lambda |v|_1|w|_1, \quad \forall \phi,v,w \in H^1_0(\Omega).
\end{align*}
We emphasize that both $\lambda \text{ and } \Lambda$ do not depend on the choice of $\phi.$ 
Therefore, the auxiliary solution $u_\phi \in H^1_0(\Omega)$ exists and satisfies the following stability property
\[|u_\phi|_1 \lesssim \|f\|_0.\]
The constant depends on the ellipticity and boundedness constants.

\subsection{Kačanov method and finite element discretization}  
  This subsection introduces the classical Kačanov scheme,  that can be viewed as a fixed point method which  will be used  as an iterative scheme to solve problem \eqref{WeakForm}. In this case, it updates a given approximation by solving a linear problem in each step. 
 
 For an initial point $u_0 \in H^1_0(\Omega)$, the traditional Kačanov scheme generates a sequence of solutions $\{u^n\}_{n\geq0}$ that approximates the solution of \eqref{PDEProblem}, defined iteratively by 
    \begin{equation*}\label{eq:tradtionalKacanove}
        -\nabla \cdot (\alpha(x,u^n)\nabla u^{n+1})=f \quad \text{in } \Omega, \quad n=0,1,2,\ldots.
    \end{equation*}
Such an iterative scheme was presented by Kačanov in the context of a variational model for plasticity problems in \cite{kacanove}.  
In weak formulation, this means that, given $u^0 \in H^1_0(\Omega)$, we  seek the solution $u^{n+1} \in H^1_0(\Omega)$ such that 
\begin{equation}\label{iterKaca}
    \mathcal{A}(u^{n};u^{n+1},v)=(f,v),\quad \forall  v \in H^1_0(\Omega) \quad n=0,1,2,\ldots.
\end{equation}

 To approximate the sequence  $\{u^n\}_{n\in \mathbb{N}}$ in practice, we need to additionally discretize the problem for numerical implementation. Consider a decomposition of the domain $\Omega$ into a partition $\mathcal{T}_H$ of simplices or quadrilaterals. Let $T$ denote the elements of $\mathcal{T}_H$,  and let $H_T$ denote the corresponding diameter, we define   \[H:=\max_{T \in \mathcal{T}_H}H_T.\] 
We assume that $\mathcal{T}_H$ is shape-regular and quasi-uniform. Let $V_H$ be the standard first-order conforming finite element subspace of $H^1_0(\Omega)$. It consists of piecewise polynomials of total degree at most 1, if $T$ is a simplex,  or of partial degree at most 1 in each variable, if $T$ is a quadrilateral.  Given $ u^n_H \in V_H$,  the Galerkin method seeks a discrete approximation  $u_H^{n+1} \in V_H$    
 such that it solves \begin{equation*}\label{disc_iterative}
    \mathcal{A}(u_H^{n};u_H^{n+1},v_H)=(f,v_H), \quad \forall  v_H \in V_H \quad n=0,1,2,\ldots.
\end{equation*}
This scheme outputs a sequence of discrete solution  $\{u_H^n\}_{n\geq 0}.$
However, in order to give an accurate discrete approximation,  the standard finite element  introduced above is required  to resolve the spatial fine-scale features  encoded in the coefficient $\alpha$ which leads to high-dimensional systems of linear equations. Therefore, in the following section,  we introduce an alternative computational multiscale method, which provides a macroscopic representation of the solution while  significantly reducing the computational cost. 
\section{Approximation based on orthogonal decomposition }\label{Section :Ortho}
 
In this section, 
we aim to incorporate the fine-scale information encoded in the coefficient $\alpha$ into the coarse scale basis of the space $V_H$. 
The construction of a multiscale space in the spirit of the LOD is based on the idea of forming the orthogonal complement of a fine-scale space $W$ (to be specified below), where orthogonality is understood with respect to an energy scalar product induced by the problem at hand. For nonlinear problems, however, such an idea does not ensure the linearity of the resulting space. As the construction of such a multiscale space is well understood in linear settings; see, for example \cite{LODFIRst},   we aim to circumvent the nonlinearity of the problem in an iterative fashion. Therefore, we first briefly recap the orthogonal decomposition explained above for our linear auxiliary problem \eqref{auxLinProb}.

\subsection{Orthogonal decomposition for the auxiliary problem}
 First, we specify the fine-scale space $W$, which mainly relies on a bounded and projective  interpolation operator \[I_H: H^1_0(\Omega)\xrightarrow{}V_H \text{ such that } I_H(v_H)=v_H \quad \forall v_H \in V_H.\]Further, we assume that $I_H$ fulfills the following (local) stability and approximation properties 
\begin{equation}\label{Interpolation_properties}
    |I_Hv|_{1,T}  + H^{-1}\|v-I_Hv\|_{0,T}  \leq C_{\operatorname{proj}} |v|_{1,N(T)} \quad \forall v \in  H_0^1(\Omega), T \in \mathcal T_H.
\end{equation}
  The subdomain $N(T)$ denotes the set of neighboring elements $T' \in \mathcal{T}_H$ of non-empty intersection with $T$. 
  Due to projective property of $I_H$, it induces the following stable decomposition \[H^1_0(\Omega)=V_H\oplus W \quad \text{with}\quad W:=\operatorname{Ker}I_H,\]
where $W$ is commonly referred  to as the fine-scale subspace. An example of projective operator satisfying \eqref{Interpolation_properties} is given by 
$I_H: H_0^1(\Omega)\xrightarrow{\Pi_H}S^1(\mathcal{T}_H)\xrightarrow{E_H}V_H.$
The operator $\Pi_H$ is defined to be the $L^2$ projection onto  $S^1( \mathcal{T}_H )$ (the space of possibly discontinuous finite element functions that are affine on each element), and $E_H$ is an averaging operator that maps $v_H \in S^1( \mathcal{T}_H )$ into $V_H$. Further examples of  projection operators satisfying  property \eqref{Interpolation_properties} can be found in \cite{Thecode}.

Now, for a given $\phi\in H_0^1(\Omega)$, we define the correction operator \[\mathcal{Q}_{\phi}:H^1_0(\Omega)\xrightarrow{} W,\]
for any $v \in H^1_0(\Omega)$, as a solution to the  following auxiliary linear  problem: 
\begin{equation}\label{CorectionProblem}
    \mathcal{A}(\phi; \mathcal{Q}_{\phi}v,w)=\mathcal{A}(\phi; v,w)\quad \forall w \in W.
\end{equation}
Note that, due to  Assumption \ref{assum1}, problem \eqref{CorectionProblem} is a linear elliptic problem and thus well-posed.  Solving problem   \eqref{CorectionProblem} for any $v \in V_H$  basically leads to the construction of a multiscale $(1-\mathcal{Q}_\phi)V_H$.  Note that the space $(1-\mathcal{Q}_\phi)V_H$ is exactly the orthogonal complement of $W$ w.r.t.\ the scalar product $\mathcal A(\phi; \cdot, \cdot)$.

Given $\phi \in H^1_0(\Omega)$, we now formulate the Galerkin method in the constructed multiscale space $(1-\mathcal{Q}_\phi)V_H$ such that it seeks  $u_{\phi,H} \in (1-\mathcal{Q}_\phi)V_H$  that solves
\begin{equation}\label{auxuliarymutliscaleproblem}
    \mathcal{A}(\phi; u_{\phi,H},v_H)=(f,v_H)\quad \forall v_H \in (1-\mathcal{Q}_{\phi})V_H.
\end{equation}
We emphasize that the linear problem \eqref{auxuliarymutliscaleproblem} will be solved in each iteration step  on the new constructed multiscale space with an updated $\phi$, see the next subsection. It is therefore essential that all constants in the following lemma, which is standard (see \cite{LODFIRst}), are independent of $\phi$.

\begin{lemma}
    Given $\phi \in H^1_0(\Omega)$, let $u_\phi$ be the solution to the auxiliary problem \eqref{auxLinProb}, and let $u_{ \phi,H}$
    be the corresponding multiscale solution of \eqref{auxuliarymutliscaleproblem}. Then the following estimate holds 
    \begin{equation}\label{linearerror}
        |u_\phi-u_{\phi,H}|_1\leq C_{\operatorname{proj}}\frac{ H}{\lambda} \|f\|_0.
    \end{equation}
\end{lemma}

\begin{proof} 
We include the proof for completeness and to highlight the independence of $\phi$.
    Observe that
    \[u_{\phi,H}=(1-\mathcal Q_\phi)I_Hu_{\phi}=(1-\mathcal Q_\phi)u_\phi.\]
    Note that  $u_\phi-u_{\phi,H}=\mathcal Q_\phi u_\phi \in W  $ and by definition  of $\mathcal{Q}_\phi$  \eqref{CorectionProblem},  
    we derive by the ellipticity assumption in \eqref{ellip}, and the interpolation estimate \eqref{Interpolation_properties} that  
        \begin{align*}  
            \lambda|u_\phi-u_{\phi,H}|_1^2 &\leq \mathcal{A}(\phi;u_{\phi}-u_{\phi,H},u_{\phi}-u_{\phi,H})=\mathcal{A}(\phi;u_{\phi},\mathcal Q_\phi u_{\phi}) \\
            &=(f,\mathcal Q_\phi u_{\phi})\\
            &\leq \|f\|_0\|\mathcal Q_\phi u_\phi\|_0=\|f\|_0\|\mathcal Q_\phi u_\phi-I_H\mathcal Q_\phi u_\phi\|_0\\
            &\leq C_{\operatorname{proj}}H\|f\|_0|\mathcal Q_\phi u_\phi|_1,
        \end{align*}
which finishes the proof.
\end{proof}

\subsection{Iterative LOD approximation}

We now introduce an iterative LOD approximation to solve the considered nonlinear problem \eqref{WeakForm} based on the Kačanov iterative method.  In our method, we emphasize that in each iteration, two linear problems are solved as follows: a construction of a multiscale space is first performed via solving the linear correction problems \eqref{CorectionProblem}, subsequently, the  auxiliary linear problem \eqref{auxuliarymutliscaleproblem} is solved with the constructed multiscale space.

Assume that $u^n_H$ is known, we first introduce the  corresponding correction operator $\mathcal{Q}_n:=\mathcal{Q}_{u^n_{H}}$ obtained via solving the linear problem \eqref{CorectionProblem} using $\phi=u^n_H$.  
Then,   we use Galerkin approximation to find the solution  $u^{n+1}_{H} \in (1-\mathcal{Q}_{n})V_H$ such that it solves 
\begin{equation*}
    \mathcal{A}(u^{n}_{H};u^{n+1}_{H},v_H)=F(v_H) \quad \forall v_H \in (1-\mathcal{Q}_{n})V_H.  
\end{equation*}
The  iterative scheme defines a sequence of LOD solutions $\{u^n_{H}\}_{n \geq 0}.$ Denote the initial   point by $\phi :=u^0_H \in H^1_0(\Omega)$,  we assume that $\phi$ is chosen such that it is a first-order approximation of $u^0$ (the initial   point of the iterative scheme \eqref{iterKaca}), i.e.,
\[|u^0-\phi|_1<C_0 H.\]
For instance, this is easily satisfied by choosing $\phi=u^0$. The iterative LOD approximation is summarized in the following algorithm \ref{algo: IterativeLOD}.

\begin{algorithm}[H]
\caption{Iterative (LOD) Approach}\label{algo: IterativeLOD}
Initialize \(  \phi=u^0_H \in H^1_0(\Omega)\)  

\Repeat{
convergence criterion is met
}{
  Construct the multiscale space \( (1-\mathcal{Q}_{\phi})V_{H} \) via solving the correction problem \eqref{CorectionProblem} for any $v\in V_H$
  \begin{equation*}
    \mathcal{A}(\phi; \mathcal{Q}_{\phi}v,w)=\mathcal{A}(\phi; v,w) \quad \forall w\in W.  
\end{equation*}  

Compute the solution $u^{n+1}_{H}$ to the linear problem:  Find $u^{n+1}_H  \in  (1-\mathcal{Q}_{\phi})V_H$ 
  \begin{equation*}
    \mathcal{A}(u^{n}_{H};u^{n+1}_{H},v_H)=(f,v_H) \quad \forall v_{H} \in (1-\mathcal{Q}_\phi)V_H.  
\end{equation*}

  Update : \( \phi \leftarrow u_{H}^{n+1} \)
}
\Return \(\{ u^n_{H}\}_{n\geq 0} \)
\end{algorithm}
 
Observe that $\phi$ and $\mathcal{Q}_{\phi}$ are updated in each iteration. Hence, we obtain a new multiscale space $(1-\mathcal{Q}_\phi)V_H$ in each iterative step. Therefore, the method remains computationally demanding. This and further steps to turn the proposed idea into a feasible and computationally efficient method will be presented in Section \ref{sec: adaptive}.   Before that, we show that the iterative LOD approximation in fact has appealing convergence properties in the next section.

\section{Error analysis  for the  iterative LOD}\label{Section: error1-2-3D}
In this section, we present an error analysis of the proposed iterative LOD method, focusing on its convergence properties. In this error analysis, the major challenge is to handle the nonlinear terms between different iterations. This requires certain additional assumptions on the regularity of the solution $u$ or $u_\phi$, which we first motivate and summarize.
\subsection{Regularity requirements } \label{RegularityRequirments}
 
Given the spatial regularity assumptions on the coefficient $\alpha \in L^\infty \text{ and } f \in L^2$, we cannot hope to obtain a solution $u$ of regularity beyond $ H^1_0(\Omega)$. Now, rather than imposing further smoothness assumptions on  $\alpha$, we allow and assume more regularity, especially higher integrability, of the right-hand side $f$. This still allows  for a broad range of coefficients,  including discontinuous ones.   Even for such rough coefficients, the additional assumptions on $f$ ensure higher regularity of the solution $u$ in the sense of improved integrability.

To motivate why this is essential for our error analysis, we consider the estimation of the $L^2$-norm \[\|(\alpha(x,u_{\phi,H})-\alpha(x,u_\phi))\nabla u_\phi\|_0,\] where $u_\phi$ denotes the solution of the linear problem \eqref{auxLinProb}, and $u_{\phi,H}$   denotes the corresponding  multiscale discrete approximation of problem \eqref{auxuliarymutliscaleproblem}.  Such terms will appear in our error analysis when considering the error between subsequent iterations. 
Upon applying H\"older’s inequality, we   seek a suitable choice of $p>2$, in particular in 2$\operatorname{D}$ and 3$\operatorname{D}$ settings,  in order to utilize the Lipschitz continuity of $\alpha$ \eqref{LipCon}. Subsequently, we  apply the Sobolev embedding of $H^1_0(\Omega)$ into $L^q(\Omega)$ (with constant $C_{q,\Omega}$) such that the following estimate holds
\begin{equation}\label{eq:assump_general}
 \begin{split}
    \|(\alpha(x,u_{\phi,H})-\alpha(x,u_\phi))\nabla u_\phi\|_{0} &\leq \|(\alpha(x,u_{\phi,H})-\alpha(x,u_\phi))\|_{L^q} \|\nabla u_\phi\|_{L^p},  \quad q:=\frac{2p}{p-2} \in (2, \infty) \\
    &\leq  {L_cC_{q,\Omega}|u_{\phi,H}-u_\phi|_1} \|\nabla u_\phi\|_{L^p}.
\end{split}   
\end{equation}
  Note that this Sobolev embedding  will limit the choices of $q$ and thereby $p$ as well. 

 It is  known via Lax-Milgram theory that the bilinear form \eqref{auxLinProb} induces a bounded linear solution operator from $H^{-1}(\Omega) \text{ into } H^1_0(\Omega)$  that admits a unique solution $u_\phi \in H^1_0(\Omega)$ for each $f \in  H^{-1}(\Omega)$. Now, we aim to investigate whether     this result can be extended into $p>2$. However, assuming  $\nabla u_\phi \in L^p \text{ for some } p>2$ is not straightforward. To ensure this, we require additional higher integrability assumption  on the right-hand side $f$, e.g., $f\in L^p(\Omega)$. We emphasize that such a condition is in fact satisfied in many applications and does not restrict the multiscale nature of the problem. 
 Prior  to further discussion, we recall the definition of negative Sobolev space $W^{-1,p}(\Omega)$. For $\Omega \subset \mathbb{R}^n, 1 \leq p\leq \infty$, we set 
  \[W^{-1,p}(\Omega) :=(W_0^{1,p^{\prime}})^{\prime},\; p^{\prime}=\frac{p}{p-1}.   \]
  The corresponding norm is defined by \[\|f\|_{W^{-1,p}(\Omega)}:=\sup_{v\in {W_0^{1,p^{\prime}}(\Omega)}, v \neq 0}\frac{\langle f,v\rangle}{\|v\|_{W^{1,p^{\prime}}(\Omega)}}.\]

The authors in \cite{Kenig} investigate the  Poisson problem $(\alpha=I)$  with Dirichlet boundary condition  and $f \in {W_0^{-1,p}(\Omega)}$. Then, for each bounded Lipchitz domain $\Omega$, they establish the existence of    
  an exponent $P$ with $P> 4, \text{ when } d=2$,  and $P> 3, \text{ when } d=3$, such    that   the Poisson problem  
admits a unique solution $u \in W^{1,p}(\Omega)$ that satisfies

\begin{equation}\label{condtionP}
    |u|_{W_0^{1,p}(\Omega)}:=\|\nabla u\|_{L_p}\leq C\|f\|_{W^{-1,P}(\Omega)},\;\;2\leq p\leq P. 
\end{equation}
The constant $C$ is independent of the choice of $f.$
In particular, for $p=P$,  the solution   satisfies the following stability
bound
\begin{equation}\label{constantCP}
    \|\nabla u\|_{L^P}\leq C_P \|f\|_{W^{-1,P}(\Omega)}.
\end{equation}
The constant $C_P$ depends on the domain $\Omega$ and the value of $P.$

The inequality $\eqref{condtionP}$   can be  also generalized to the  diffusion problem where  $\alpha \in L^\infty$ whenever $f \in W^{-1,p}(\Omega)$, but  further conditions on the contrast of the diffusion coefficient $\alpha$ might be required. For further detailed proofs for  general diffusion elliptic problems, we refer to  \cite{BrennerScott,Meyer}, where such an inequality is shown using a perturbation argument. We also refer to  \cite{LP,Kenig} for  further details and proofs.   The following proposition introduces the assumptions required  to ensure that $\eqref{condtionP}$ holds for general diffusion problems, see \cite{LP}.
\begin{prop} Let $f$ and $\Omega$ be such that for some $P>2$ and for a constant $C_P$, the solution $u \in H^1_0({\Omega})$ of the Poisson problem with Dirichlet boundary satisfies the stability estimate  $\eqref{constantCP}$ whenever $f \in W^{-1,P}(\Omega)$. If $\alpha \in  L^\infty$ satisfies the boundedness and ellipticity condition, then the solution $u_\phi \in H^1_0(\Omega)$ of the linear auxiliary problem \eqref{auxLinProb} 
    satisfies 
    \begin{equation}\label{ConditionPGen}
        \|\nabla u_\phi\|_{L^p}\leq c_p\|f\|_{W^{-1,P}(\Omega)}
    \end{equation}
    given that $2\leq p\leq P$ with the constant 
    \begin{equation}\label{theconstant}
        c_p:=\frac{C^{\eta(p)}_P}{\Lambda(1-C^{\eta(p)}_P(1-\frac{\lambda}{\Lambda}))}.
    \end{equation}
The term  ${\eta(p)}$ is defined in \cite{LP} as \[{\eta(p)} :=\frac{\frac{1}{2}-\frac{1}{p}}{\frac{1}{2}-\frac{1}{P}}\]
which can be seen as   a convex  function and has a minimum $({\eta(p)} =0)$ at $p=2$. 
\end{prop}

\paragraph{Regularity requirements in $\operatorname{ 1D \text{ and } 2D}$:}
\label{Rem:Cp_definition}
Recalling \eqref{eq:assump_general}, we return to the question which values of $q$ are allowed to have the Sobolev embedding of $H^1_0(\Omega)$ into $L^q(\Omega)$. In   2$\operatorname{D}$, any  $q<\infty$ is admissible according to \cite{embdeding}, which in turn means that we need $p>2$.
Therefore,  we require in $\operatorname{  2D}$ that  $f\in L^{2+\epsilon}$  to ensure that $u_\phi \in W^{1,2+\epsilon}$, with $\epsilon>0$. Hence, we obtain via \eqref{theconstant} 
\begin{equation}\label{c_epsilon}
    c_{2+\varepsilon} \approx 1/\lambda.
\end{equation}
Then, as a result to the above explanation,   the  solution satisfies $\nabla u_\phi \in L^{2+\epsilon}$, and the following estimate holds\[\| (\alpha(x,u_{\phi,H})-\alpha(x,u_\phi))\nabla u_\phi\|_{0}\leq\| \alpha(x,u_{\phi,H})-\alpha(x,u_\phi)\|_{L^q}\|\nabla u_\phi\|_{L^{2+\epsilon}}, \quad  q=\frac{2(2+\epsilon)}{\epsilon}.\]
Now, recalling \eqref{eq:assump_general}, we conclude that 
\begin{equation}\label{2d_assumptions}
    \|(\alpha(x,u_{\phi,H})-\alpha(x,u_\phi))\nabla u_\phi\|_{0} \leq L_c c_{2+\varepsilon}\| u_{\phi,H}-u_\phi\|_{L^q}\|f\|_{L^{2+\epsilon}}\leq L_c c_{2+\varepsilon} C_{q,\Omega}| u_{\phi,H}-u_\phi|_1\|f\|_{L^{2+\epsilon}}.
\end{equation}
Observe that $q\xrightarrow{}\infty  \text{ as } \epsilon \xrightarrow{}0$. Hence, $\epsilon$ must be chosen such that \eqref{c_epsilon} is satisfied and $q \ll \infty$, thereby ensuring that   $C_{q,\Omega}$ remains bounded, particularly in $\operatorname{2D}$. 
Note that in the one-dimensional case, no additional assumptions on the right-hand side are required. Since $H^1_0(\Omega)\hookrightarrow L^{\infty}(\Omega)$, it follows from \eqref{2d_assumptions} that  $ f \in L^2(\Omega)$ $(\epsilon=0)$ is sufficient, implying $ u_\phi \in H^1_0(\Omega)$.
\begin{remark}[$\operatorname{3D}$  regularity requirement]\label{rem: 3D assumptions}  
     In   $\operatorname{3D}$,  the  Sobolev embedding theorem gives $H^1_0(\Omega)\hookrightarrow L^q(\Omega) \; \forall q \in [2,6]$, which means that we need at least $p=3$ in \eqref{eq:assump_general}. From \cite{BrennerScott, Meyer}, we infer that condition   
\begin{equation}\label{lessone}
    C_P(1-\frac{\lambda}{\Lambda})<1  
\end{equation} 
in \eqref{theconstant} must be satisfied along with $f \in L^3(\Omega)$ to ensure that the weak solution $u_\phi \in W^{1,3}(\Omega).$ This assumption limits the applicability in $\operatorname{3D}$ to diffusion coefficients with small or moderate contrast. For this reason, we focus our following discussion more on the one- and two-dimensional cases.  In $3D$,  the upper bound constant of \eqref{ConditionPGen} is given as  \[c_p \approx \frac{C_P}{\Lambda(1-C_P(1-\frac{\lambda}{\Lambda}))}.\] 
Recalling \eqref{eq:assump_general}, we obtain the bound
\begin{equation*}\label{3Dinequality}
        \begin{split}
          \| (\alpha(x,u_{\phi,H})-\alpha(x,u_\phi))\nabla u_\phi\|_{0}&\leq    L_c\| u_{\phi,H}-u_\phi\|_{L^6}\|\nabla u_\phi\|_{L^3}\\
          &\leq L_cc_3C_{6,\Omega} |u_{\phi,H}-u_\phi|_{1}\|f\|_{L^3.} 
      \end{split}
  \end{equation*}
\end{remark}
 
 \subsection{Error analysis}\label{subsec: error analysis}
In this subsection, we analyze and estimate the error $|u^n-u^n_H|_1, \; n=0,1,2,\ldots $, where $u_H^n$  denotes the multiscale approximation solution obtained via the iterative LOD, see Algorithm \ref{algo: IterativeLOD}, and  $u^n$ is the weak solution of \eqref{iterKaca}.
 \begin{theorem}\label{pro:iterative convergence}
     
    Let  $\{u^n\}_{n \in \mathbb{N}}$ and $\{u_H^n\}_{n\in \mathbb{N}}$ be  the output sequence of weak iterative solutions and multiscale iterative solutions, respectively. Assume that $f \in L^{2+\epsilon}(\Omega) $ in $\operatorname{2D}$, with $\epsilon>0$. In $\operatorname{1D}$, we formally  set $\epsilon=0$. Accordingly,   let $\nu  $  be defined 
     as 
\begin{equation}\label{definition fo nu}
    \nu := 
\frac{L_c c_{2+\epsilon}C_{q,\Omega} \|f\|_{L^{2+\epsilon}}}{   \lambda }. 
\end{equation}
If $\nu <1$,  the output sequence of solutions  $\{u^n\}_{n \geq 0}$ of the Kačanov scheme introduced in \eqref{iterKaca} converges to the  weak solution $u$ of problem \eqref{WeakForm}. Moreover,   the following error estimate holds true for all $n \in \mathbb{N}$

\begin{equation}\label{errorbound}
    |u^n-u_{H}^n|_1\leq \nu^n |u_H^0-u^0|_1+ \frac{1}{1-\nu} \frac{C_{\operatorname{proj}}}{\lambda} H \|f\|_0.
\end{equation}

 \end{theorem}

 \begin{proof}  We first prove the convergence of  the sequence $\{u^n\}_{n \geq 0}$, i.e.,  $u^n \xrightarrow{}u$. Via Assumption \ref{assum1} on Lipschitz continuity and the ellipticity of $\alpha$, and  by the estimate \eqref{2d_assumptions},   we obtain the following estimate  
 \begin{align*}
     \lambda |u-u^n|^2_1& \leq \mathcal{A}(u^{n-1};u^n-u,u^n-u)=\mathcal{A}(u^{n-1};u^n,u^n-u)-\mathcal{A}(u^{n-1};u,u^n-u)\\
     &=\mathcal{A}(u;u,u^n-u)-\mathcal{A}(u^{n-1};u,u^n-u)=( (\alpha(u^{n-1})- \alpha(u)) \nabla u, \nabla (u^n-u)) \\
     &\leq \| (\alpha(u^{n-1})- \alpha(u)) \nabla u\|_0\|\nabla (u^n-u)\|_0\\
     & \leq L_c c_{2+\epsilon}C_{q,\Omega} \|f\|_{L^{2+\epsilon}}  |u-u^{n-1}|_1 |u^n-u|_1. 
\end{align*}
Inductively,  this yields
\[|u-u^n|_1\leq \nu^n|u-u^0|_1.\]
 Hence, by the assumption that $\nu<1$, the sequence $u^n \xrightarrow{} u$.

 Now, we aim to bound  the error  associated  with the iterative LOD technique.   For a fixed $n \in \mathbb{N}$, we define $w^n \in H_0^1(\Omega)$  as the solution of the  auxiliary linear problem 
   \begin{equation*} \label{aux}
       \mathcal{A}(u_H^{n-1};w^n,v)=F(v)  \quad  \forall v \in H^1_0.
   \end{equation*}
Following the discussion in the previous Subsection \ref{RegularityRequirments}, we note that   $w^n$ satisfies the following stability \[\|\nabla w^n\|_{L^{2+\epsilon} }\leq c_{2+\epsilon}\|f\|_{L^{2+\epsilon} }.\]
We emphasize that the consonant $c_{2+\epsilon}$ does not depend on the choice of  $f$ or  the linearization point $u_H^{n-1}$.  Now, by the triangle inequality, we have that
\[|u^n-u^n_H|_1\leq |u^n-w^n|_1 +|w^n-u^n_H|_1.\]
First, we bound the error $w^n-u^n_H$ between the LOD and the weak solution of the  linearized problem at $u_H^{n-1}$. It follows by \eqref{linearerror} that

\begin{equation}\label{estimet1}
    |w^n-u^n_H|_1\leq \frac{C_{\operatorname{proj}}}{\lambda} H \|f\|_0.
\end{equation}
 Second, the given  ellipticity assumption \eqref{ellip}, Lipschitz continuity \eqref{LipCon},  and the estimate  (\ref{2d_assumptions}),  we obtain that
  \begin{equation}\label{replace}
        \begin{split}
          \lambda|w^n-u^n|_1^2&\leq \mathcal{A}(u^{n-1}; w^n-u^n, w^n-u^n)\\
          &\leq \mathcal{A}(u^{n-1}; w^n, w^n-u^n)-\underbrace{\mathcal{A}(u^{n-1};u^n, w^n-u^n)}_{=F(w^n-u^n)=\mathcal A(u_H^{n-1}; w^n, w^n-u^n)}\\
        & =(( \alpha(u^{n-1})- \alpha(u_H^{n-1})) \nabla w^n,\nabla (w^n-u^n)) \\
        &\leq  \|(\alpha(u^{n-1})-\alpha(u^{n-1}_H))\nabla w^n\|_{0}|w^n-u^n|_1 \;\\
        &\leq 
        L_c c_{2+\varepsilon} C_{q,\Omega}  \|f\|_{L^{2+\epsilon}} |u^{n-1}-u^{n-1}_H|_1 |w^n-u^n|_1.
      \end{split}
  \end{equation}
Combining both estimate \eqref{estimet1} and \eqref{replace}, and proceeding inductively using $\nu<1$ together with the geometric series argument for the second term, we obtain the following estimate     
\begin{align*}
    |u^n-u_H^n|_1& \leq|u^n-w^n|_1+|w^n-u_H^n|_1 \\
    & \leq \nu  |u^{n-1}_H-u^{n-1}|_1 +\frac{C_{\operatorname{proj}}}{\lambda}  H \|f\|_0\\
    & \leq \nu^n |u^0_H-u^0|_1+\frac{1}{1-\nu}\frac{C_{\operatorname{proj}}}{\lambda}  H \|f\|_0.\qedhere
\end{align*}

 \end{proof}

\begin{remark}\label{rem:3D} In the same way, the  results and the proof above extend to  a $\operatorname{3D} $ domain. However, it is necessary to assume that $f\in L^3(\Omega)$  and the bound \eqref{lessone} on the contrast. Consequently, under these assumptions, the contraction constant $\nu $ is redefined  as 
\begin{equation}\label{definition fo nu 3D}
    \nu := \frac{L_c c_3 C_{6,\Omega} \|f\|_{L^3}}{   \lambda }   \text{ in } \operatorname{3D},
\end{equation}
and assuming again $\nu<1$, the estimate \eqref{errorbound} holds for the $\operatorname{
3D}$ setting. The result follows directly by applying the same sequence of steps as established in the preceding proof for $\operatorname{1D \text{ and } 2D}.$
    
\end{remark}

\begin{remark}
    All the presented theory above can be extended from the regularity of the right-hand side $f \in L^p(\Omega)$ -- with the values of $p$ as specified above --  to the assumption $f \in W^{-1,p}(\Omega) \cap L^2(\Omega).$
\end{remark}

Although the iterative (LOD) exhibits linear convergence, the repeated construction of the multiscale space in each iteration  is computationally expensive. This  motivates an adaptive approach for the update of the multiscale space. Maintaining  the same theoretical assumptions and regularity requirements  as presented above, we next discuss the practical realization of the iterative LOD approach, including localization and the explained adaptivity.

\section{Adaptive Iterative LOD approximation}\label{sec: adaptive}
As already discussed at the end of Section \ref{Section :Ortho}, the proposed iterative multiscale scheme is not yet feasible in practice. First, the linear problems to compute the correction $\mathcal Q_{\phi}$ are global fine-scale problems, which are as costly as one iteration step for the original problem. Second, the multiscale space $(1-{\mathcal{Q}}_\phi)V_H$ is computed  anew in each iteration which means additional computational overhead. In this section, we address both challenges by a localization and adaptive strategy (over the iterations) for the corrector computations.
Finally, we show how these modifications affect the convergence and error estimation result.

 \subsection{Localization}\label{subsec: locaiztion}
 Since we have to solve a linear elliptic problem in each step of our iterative LOD algorithm, we follow  by now the standard localization procedure of the correction problem~\eqref{CorectionProblem} as presented in \cite{LODFIRst}. To be self-contained,  we briefly summarize the key ingredients to localize the auxiliary global problem \eqref{auxLinProb}. We emphasize that other more recent localization strategies  could be used as well; see, for example,  the SLOD method in \cite{SLOD}.    
 Let $N^k(T)$ denote the $k$-layer patch of neighboring elements  to $T$ that is inductively defined for $k\geq 1$ by
 \[N^k(T)=N(N^{k-1}(T)), \quad k\geq 2, \qquad \text{ and  }N^1(T)=N(T).\]
 The quasi-uniformity of $\mathcal{T}_H$ ensures that the number of elements that belong to $N^k(T)$ is bounded by a constant $C_\mathrm{ol}$ that only depends on $k$ in a polynomial manner, i.e., 
\begin{equation}\label{Uniform_regu}
\max_{T \in \mathcal{T}_H}|\{T' \in \mathcal{T}_H:T' \in N^k(T)\}| \leq C_\mathrm{ol}.
\end{equation}
Given $\phi \in H^1_0(\Omega)$, the local and truncated element-based corrector 
\[ \mathcal{Q}_{\phi,T}^k:V_H \xrightarrow{}W(N^{k}(T)):=\{ w \in W : \text{supp}(w)  \subseteq\overline{N^k(T)}\}, \] is defined as the solution of the following local correction problem 

\begin{equation}\label{Corr}
    \mathcal{A}_{N^k(T)}(\phi; \mathcal{Q}_{\phi,T}^kv_H, w)=\mathcal{A}_{T}(\phi;v_H, w), \quad \forall w \in W(N^{k}(T)).
\end{equation}
Here, $\mathcal{A}_{D}$  denotes the restriction of  the bilinear form $\mathcal{A}(\phi;\cdot,\cdot)$ to the subdomain $D\subset \Omega$. 
Now, we define the global truncated  correction operator as \[\mathcal{Q}_{\phi}^k:V_H \xrightarrow{} W, \quad \mathcal{Q}_{\phi}^k=\sum_{T \in \mathcal{T}_H } \mathcal{Q}_{\phi,T}^k.\] 
Note that the corrector problems~\eqref{Corr} are still posed on infinite dimensional subspaces $W(N^{k}(T))$, which need to be discretized. This is accomplished as usual by introducing a fine-scale mesh $\mathcal{T}_h$ of $\Omega$, such that $h\ll H$ resolves the multiscale features of $\alpha$. 

The following proposition demonstrates that the error between the corrector operator and its localized counterpart exhibits an exponential decay with respect to the oversampling parameter $k$, see \cite{LODFIRst}.
\begin{prop}\label{my_proposition}
   Let $\phi\in H^1_0(\Omega)$,  and let Assumptions~\eqref{ellip} be satisfied. Let $\mathcal{Q}_\phi$ be the linear corrector defined in~\eqref{CorectionProblem} and its localized version $\mathcal{Q}^k$ defined in~\eqref{Corr}. There exists a constant $0<\beta<1$ such that for any $v_H \in V_H$, it holds that
   \begin{equation}\label{eq: decayResult}
    |(\mathcal{Q}_\phi-\mathcal{Q}_\phi^k)v_H |_1\lesssim C_\mathrm{ol}^\frac{1}{2} \beta^k |v_H|_1,
   \end{equation}
   where $C_{\operatorname{ol}}$ is the constant in~\eqref{Uniform_regu}. Moreover,  we underline that the constants in the estimate do not depend on  the choice of $\phi$.
\end{prop}

 \subsection{Adaptive strategy}\label{subsec:errorIndic}
Motivated by the technique presented in \cite{adaptive2,adaptive1} for perturbed elliptic PDE, and in \cite{MAir&barbara} for nonlinear Helmholtz equations, we aim to improve the efficiency of the iterative LOD by recomputing the correction operators only where and when it is necessary. We emphasize that the localization of the corrector computations accomplished above is a necessary requirement for the following adaptive strategy.
The key idea is to derive an error indicator -- following \cite{adaptive2,adaptive1} --  for the correction operators when the local corrector problems \eqref{Corr} are computed at two different linearization points $\psi$ and $\xi$. The proposed indicator provides a criterion at the local level for identifying where updating the correction operators is required between successive iterations.
 
Given $\psi, \xi \in H^1_0(\Omega)$, let $\mathcal{Q}^k_{\psi,T}$ and $\mathcal{Q}^k_{\xi,T}$ be the corresponding correction operators computed by solving the linear problem \eqref{Corr} at  $\psi, \xi$ respectively.  We aim to compute the error between  $\mathcal{Q}^k_{\psi,T}$ and $\mathcal{Q}^k_{\xi,T}$, assuming $\psi, \xi$, and  $\mathcal{Q}^k_{\psi,T}$ are available and $\mathcal{Q}^k_{\xi,T}$ is not. In the context of our method, at iteration $n$,  $\psi$ should be seen as the multiscale solution computed at iteration $n-2$, and the corresponding local corrector $\mathcal{Q}^k_{\psi,T}$ is computed at iteration step $n-1$, whereas $\xi$ corresponds to the solution of the previous iteration $n-1 $, and we need to determine whether the correction $\mathcal{Q}^k_{\xi,T}$ needs to be   recomputed at step $n$. 
 
\begin{lemma}    Assume that $\psi, \xi \in H^1_0(\Omega) $, and the local corrector $\mathcal{\mathcal{Q}}_{\psi,T}^k$ are available. Then, 
for all $v_H \in V_H$, the following bound holds 
\[|(\mathcal{\mathcal{Q}}_{\xi,T}^k-\mathcal{\mathcal{Q}}_{\psi,T}^k)v_H|_1 \leq  \lambda^{-1}  e_{T}(\xi,\psi)|v_H|_{1,T}.\]
where 
\begin{equation}\label{eq: errorindicator}
    e_{T}(\xi,\psi)^2:=\sum_{\substack{T' \in N^k(T) }} \|\alpha(\xi)-\alpha(\psi)\|^2_{L^\infty(T')} \cdot \max_{w_{|_T}: w\in V_H } \frac{\|(\chi_T \nabla w-\nabla \mathcal{\mathcal{Q}}_{\psi,T}^kw)\|^2_{0,T' }}{\|\nabla w\|^2_{0,T}}.
\end{equation}
Moreover, we have that
\[|(\mathcal{\mathcal{Q}}_{\xi}^{k}-{\mathcal{Q}}_{\psi}^{k})v_H|_1 \leq \lambda^{-1}  {C_{\operatorname{ol}}^{\frac 12}} \Bigl(\max_{T \in \mathcal{T}_H} e_{T}(\xi,\psi)\Bigr)\, |v_H|_1.\]
\end{lemma}

\begin{proof}
 Let $v_H \in V_H  \text{ and define } e := (\mathcal{\mathcal{Q}}_{\xi,T}^k-\mathcal{\mathcal{Q}}_{\psi,T}^k)v_H$. By  the ellipticity assumption  \eqref{ellip}, and the orthogonality in  \eqref{Corr}, we obtain 
    \begin{align*}
       \lambda |e|_1^2 &\leq \mathcal{A}_{N^k(T)}(\xi;e,e)=  (\alpha(\xi) \nabla(\mathcal{\mathcal{Q}}_{\xi,T}^k-\mathcal{\mathcal{Q}}_{\psi,T}^k)v_H, \nabla e)_{N^k(T)}\\
        &=(\alpha(\xi) \nabla \mathcal{\mathcal{Q}}_{\xi,T}^k v_H, \nabla e)_{N^k(T)} + (\alpha(\psi) \nabla \mathcal{\mathcal{Q}}_{\psi,T}^kv_H, \nabla e)_{N^k(T)}-(\alpha(\psi) \nabla v_H, \nabla e)_{T}\\
        &\quad -(\alpha(\xi) \nabla \mathcal{\mathcal{Q}}_{\psi,T}^kv_H, \nabla e)_{N^k(T)}                    \\
        &\leq ((\alpha(\xi)-\alpha(\psi))\nabla v_H, \nabla e)_{T} + ((\alpha(\psi)-\alpha(\xi)) \nabla \mathcal{\mathcal{Q}}_{\psi,T}^kv_H, \nabla e)_{N^k(T)}\\
        & \leq  \|(\alpha(\xi)-\alpha(\psi))(\chi_T \nabla v_H -\nabla \mathcal{Q}^{k}_{\psi,T}v_H)\|_{0,N^k(T)} |e|_{1,N^k(T)}.
    \end{align*}
    We conclude that \[\lambda |e|_1 \leq   \|(\alpha(\xi)-\alpha(\psi))(\chi_T \nabla v_H -\nabla {\mathcal{Q}^{k}}_{\psi,T}v_H)\|_{0,N^k(T)}.\]
We continue to further bound the term on the right-hand side by observing that
\begin{align*}
    \lambda^2 |e|_1^2 &\leq \|(\alpha(\xi)-\alpha(\psi)) (\chi_T \nabla v_H-\nabla \mathcal{Q}^k_{\psi,T}v_H)||_{0,N^k(T)}^2 \\
     &\leq \sum_{\substack{T' \in   N^k(T)   }} \max_{w|_T: w\in V_H }\frac{\|(\alpha(\xi)-\alpha(\psi)) (\chi_T \nabla w-\nabla  \mathcal{\mathcal{Q}}_{\psi,T}^kw)\|^2_{0, T'}}{\|\nabla w\|_{0,T}^2}\|\nabla v_H\|^2_{0, T} \\
    & \leq  \sum_{\substack{T' \in   N^k(T)   }} \|\alpha(\xi)-\alpha(\psi)\|^2_{\infty, T'} \cdot \max_{w|_T: w\in V_H} \frac{\|(\chi_T \nabla w-\nabla \mathcal{\mathcal{Q}}_{\psi,T}^kw)\|^2_{0,T' }}{\|\nabla w\|^2_{0,T}}\|\nabla v_H\|^2_{0,T}\\
    &\leq e_{T}(\xi,\psi)^2 \|\nabla v_H\|^2_{0,T}.
\end{align*}
Based on the definition of the local correction operator, it can be represented as the sum of the element correctors that are supported in the oversampling patch $N^k(T)$. Hence, we conclude via \eqref{Uniform_regu} that the following estimate holds 

\begin{equation}\label{eq: correctordifferentpoints}
    \begin{split}
        |(\mathcal{Q}^k_{\xi}-\mathcal{Q}^k_{\psi})v_H|^2_1 &\leq C_{\operatorname{ol}}  \sum_{T \in {\mathcal{T}}_H} |(\mathcal{Q}^k_{\xi, T}-\mathcal{Q}^k_{\psi, T})v_H|_1^2 \\
        &\leq C_{\operatorname{ol}} \lambda^{-2}(\operatorname{max}_{T \in \mathcal{T}_H}e_{T}(\xi,\psi))^2\|\nabla v_H\|_0^2.
    \end{split}
\end{equation}
We refer to \cite{adaptive2, adaptive1, MAir&barbara} for a similar calculation and further details.
\end{proof} 
\begin{remark}
The   part  $\max_{w|_T: w\in V_H } \frac{\|(\chi_T \nabla w-\nabla \mathcal{\mathcal{Q}}_{\psi,T}^kw)\|^2_{0,T' }}{\|\nabla w\|^2_{0,T}}$ of the indicator  corresponds to the maximum  eigenvalue of a low dimensional eigenvalue problem. For further details and discussion of the implementation of the error indicator, we refer to \cite{adaptive2,adaptive1}.  
\end{remark}
Altogether, we propose the following adaptive iterative LOD algorithm \ref{AlgorithofAdaptivity} for approximating the solution of  problem \eqref{WeakForm},  which   utilizes the error indicator \eqref{eq: errorindicator} and generates a sequence of adaptive multiscale solutions. We emphasize that, at iteration  $n=0$,  all the correctors are initially computed using the initial point $u^0_H$. In subsequent iterations, in contrast to the previous Algorithm \ref{algo: IterativeLOD}, we do not recompute all the correctors. Instead, the error indicator is locally evaluated to identify the elements $T$, for which $e_T^n>\operatorname{Tol}$. These elements are then collected in $M_n$. The correctors $Q^k_{u_H^{n-1},T}$ associated with elements in the set $M_n$ are recomputed accordingly, whereas for all other elements the correctors from the previous iteration are re-used. Altogether, this forms the corrector $\tilde{Q}_{n-1}^k$ at iteration $n$. 
 
\begin{algorithm}[H]
\caption{Adaptive Iterative Localized Orthogonal Decomposition (LOD) Approach}
\label{AlgorithofAdaptivity}
\KwIn{\( u^0_H \in H^1_0(\Omega)\) (initial guess), $\operatorname{Tol}$ (for error indicator updates), $\operatorname{tol}$ (for the convergence of the iteration scheme),  and oversampling parameter $k$ }
\For{$n=0,1,2,\ldots, N_{\operatorname{max}}
$}{
    \eIf{$n = 0$}{
    $e_T^n\leftarrow\infty \;\; \forall T \in \mathcal{T}_H$
    
        Construct multiscale space \( (1-\mathcal{Q}^k_{u^0_H })V_{H} \) via solving the correction problem \label{Construct}
  \begin{equation*}
   \mathcal{A}_{N^k(T)}(u^0_H; \mathcal{\mathcal{Q}}_{u^0_H,T}^kv_H, w)=\mathcal{A}_{T}(u^0_H ;v_H, w), \quad\forall w \in W(N^{k}(T)).
\end{equation*}

  Compute the solution $u^1_{H}$ to the linearized problem: Find $u^1_H  \in  (1-\mathcal{Q}^k_{u^0_H} )V_H$  \label{compute}
  \begin{equation*}
   \mathcal{A}(u^0_H ;u^1_{H},v_H)=(f,v_H) \quad \forall v_H \in (1-\mathcal{Q}^k_{u^0_H} )V_H.
\end{equation*}
    }{
    \For{$ T \in \mathcal{T}_H$}{
       Compute the  error indicator  \label{compute error indicator}
         $e_T^n \leftarrow e_{k,T}({{u}_H^{n-1}}, {u}_H^{n-2}$)  \quad \text{( $\tilde{\mathcal{Q}}^k_{{{u}_H^{n-1}},T}$}  is not available)
         
        $M_n\leftarrow \{T \in \mathcal{T}_H : e_T^n >\operatorname{Tol}\}$
         }
        Recompute and update the correctors corresponding to the elements $ T \in M_n$ \label{recompute}
        
        Solve the Galerkin problem:  Find $\tilde{u}^n_{H,k} \in (1-\tilde{\mathcal{Q}}_{n-1}^k)V_H$ such that   \label{solve_problem}
        \[\mathcal{A}(\tilde{u}^{n-1}_{H,k};\tilde{u}^n_{H,k};\tilde{v}_{H,k})=(f,\tilde{v}_{H,k})\ \forall \tilde{v}_{H,k} \in (1-\tilde{\mathcal{Q}}_{n-1}^k)V_H.\]
    
}
if $|\tilde{u}^{n}_{H,k}-\tilde{u}_{H,k}^{n-1}|< \operatorname{tol}$, stop

}
\Return $\{\tilde{u}^n_{H,k}\}_{n \geq 0}$
\end{algorithm}

\subsection{Error analysis for the  adaptive iterative method}
 Here, we extend the a priori error analysis presented in Section \ref{Section: error1-2-3D} to investigate the error associated with the adaptive iterative LOD  technique. 
The result presented in this section is applicable for a bounded Lipschitz domain $\Omega \in \mathbb{R}^d, d \in \{1,2,3\}$. In particular, we utilize  Theorem \ref{pro:iterative convergence},  that presents the analysis for one- and two- dimensional cases,  together with Remark \ref{rem:3D},  that extends the analysis into three-dimensional setting in the previous Subsection \ref{subsec: error analysis}.

\begin{theorem}\label{th:maintheorem}
    Given $k \text{ and } \operatorname{Tol}$,  let $\{\tilde{u}^n_{H,k}\}_{n \in \mathbb{N}}$ denote the  sequence generated by Algorithm \ref{AlgorithofAdaptivity} with initial point $\tilde{u}^0$, and let $\{u^n\}_{n \geq 0}$ denote the sequence of solutions of the Kačanov scheme in \eqref{iterKaca} initialized  at $u^0$. In addition, we assume that $f \in L^{2+\epsilon}(\Omega)$, where $\epsilon>0$  in two-dimensional setting, while  $\epsilon=0$ in one dimensional case, whereas $f \in L^{3}(\Omega)$ together with inequality \eqref{lessone} in three dimensions.  Given the contraction constant  $\nu$   defined in \eqref{definition fo nu} for $\operatorname{1D}$ and $\operatorname{2D}$ domains,  and in \eqref{definition fo nu 3D} for a $\operatorname{3D}$ domain, then, if $\nu<1$, it holds that 
\[|\tilde{u}^n_{H,k} -u^n|_1 \lesssim \nu^n |\tilde{u}^0 -u^0|_1+ \frac{1}{1-\nu} (1+\frac{\Lambda}{\lambda})(C_{\operatorname{proj}} H  + \Lambda C_{\operatorname{ol}}^{\frac{1}{2}}( \beta^{k} + \operatorname{Tol}))\lambda^{-1} \|f\|_0.\]
  The constant hidden in $\lesssim$ is independent of the multiscale variations of $\alpha$, but depends on the constants in \eqref{ellip}.

\begin{proof} 
 
For a fixed $n \in \mathbb{N},$ we define the following auxiliary solution  $w  \in H^1_0(\Omega)$ that solves  
      \begin{equation} \label{aux_adaptive}
      \mathcal{A}(\tilde{u}_{H,k}^{n-1};w^n,v)=F(v)  \quad  \forall v \in H^1_0(\Omega).
   \end{equation}
By  the triangle inequality, we have   \[|\tilde{u}^n_{H,k} -u^n|_1 \leq |\tilde{u}^n_{H,k} -w^n |_1 +|w^n  -u^n|_1.\]
First,  we start by estimating the term $|w^n -u^n|_1$. Following the same steps used to prove \eqref{replace} for the iterative approach, and given \eqref{definition fo nu} for $\operatorname{1D}$ and $\operatorname{2D}$,  as well as  \eqref{definition fo nu 3D} for $\operatorname{3D}$    under the assumption $\nu<1$, we conclude that   \begin{equation}\label{eq:1stestimate}
    |w^n -u^n|_1\leq \nu |u^{n-1}-\tilde{u}_{H,k}^{n-1}|_1.
\end{equation}
Now, we estimate the first term $|\tilde{u}^n_{H,k} -w^n |_1$. Note that $\tilde{u}^n_{H,k}$ can be equivalently rewritten  as $\tilde{u}^n_{H,k}=(1-\tilde{\mathcal{Q}}^k_{n-1})I_H\tilde{u}^n_{H,k}$. Set \[z:=w^n-\tilde{u}^n_{H,k}, \text{ and } z_{H,k} :=(1-\tilde{\mathcal{Q}}^k_{n-1})I_Hz.\]    Observe that $z-z_{H,k}= w^n-(1-\tilde{\mathcal{Q}}^k_{n-1})I_Hw^n \in W$.  
   Consequently, by means of  the orthogonality defined in \eqref{CorectionProblem} at ${\tilde{u}^{n-1}_H}$,  the decay property in \eqref{eq: decayResult}, and the estimates \eqref{eq: correctordifferentpoints} and \eqref{aux_adaptive},  we obtain  
   \begin{equation}\label{eq:middlestep}
    \begin{split}
         \lambda| z-z_{H,k}|_1^2 &\leq\mathcal{A}(\tilde{u}^{n-1}_{H,k};w^n -(1-\tilde{\mathcal{Q}}^k_{n-1})I_Hw^n, z-z_{H,k}) \\
         &= (f,z-z_{H,k})-\mathcal{A}(\tilde{u}^{n-1}_{H,k};(1- \mathcal{Q}^k_{\tilde{u}^{n-1}_{H,k}})I_Hw^n, z-z_{H,k}) \\
         &\quad - \mathcal{A}(\tilde{u}^{n-1}_{H,k};( \mathcal{Q}^k_{\tilde{u}^{n-1}_{H,k}}-\tilde{\mathcal{Q}}^k_{n-1})I_Hw^n, z-z_{H,k}) \\
         & = (f,z-z_{H,k})-\mathcal{A}_{\operatorname{}}(\tilde{u}^{n-1}_{H,k};(\mathcal{Q}_{\tilde{u}^{n-1}_H}- \mathcal{Q}^k_{\tilde{u}^{n-1}_H})I_Hw^n, z-z_{H,k}) \\
         &\quad - \mathcal{A}(\tilde{u}^{n-1}_{H,k};(\mathcal{Q}^k_{\tilde{u}^{n-1}_H}-\tilde{\mathcal{Q}}^k_{n-1})I_Hw^n, z-z_{H,k}) \\
         &\lesssim C_{\operatorname{proj}} H\| f\|_0 |z-z_{H,k}|_1+ \Lambda(C_{\operatorname{ol}}^{\frac{1}{2}}\beta^k+ C_{\operatorname{ol}}^{\frac{1}{2}} \operatorname{Tol}) \|f\|_0 |z-z_{H,k}|_1.
    \end{split}
     \end{equation}
Now, we seek to bound $|z_{H,k}|_1$. Observe the following Galerkin orthogonality   \[\mathcal{A}(\tilde{u}^{n-1}_{H,k};z,z_{H,k})=0.\]
Therefore, via the assumption \eqref{ellip}, we conclude that 
  \begin{align*}
      \lambda |z_{H,k}|_1^2\leq \mathcal{A}(\tilde{u}^{n-1}_{H,k};z_{H,k},z_{H,k})& =\mathcal{A}(\tilde{u}^{n-1}_{H,k};z_{H,k}-z,z_{H,k})\\
      & \leq \Lambda |z-z_{H,k}|_1 |z_{H,k}|_1
  \end{align*}
 and, thereby, we obtain  \[|z_{H,k}|_1\leq \frac{\Lambda}{\lambda}|z_{H,k} -z|_1.\]
 Hence, using the above estimate and the estimate \eqref{eq:middlestep}, we conclude  that 
 \begin{equation}\label{eq:2ndestimate}
     \begin{split}
         |z|_1&=|w^n-\tilde{u}^n_{H,k}|_1 \leq |z-z_{H,k}|_1 +|z_{H,k}|_1\leq  (1+\frac{\Lambda}{\lambda})|z-z_{H,k}|_1 \\
         &\lesssim (1+\frac{\Lambda}{\lambda})\left(C_{\operatorname{proj}} H  + \Lambda(C_{\operatorname{ol}}^{\frac{1}{2}}\beta^k + C_{\operatorname{ol}}^{\frac{1}{2}} \operatorname{Tol}) \right) \lambda^{-1}\|f\|_0.
     \end{split}
 \end{equation}
Combining  the estimates \eqref{eq:1stestimate} and  \eqref{eq:2ndestimate}, we obtain that
\begin{align*}
  |\tilde{u}^n_{H,k} -u^n|_1 &\leq |w^n -u^n|_1+|\tilde{u}^n_{H,k} -w^n|_1   \\
   &\lesssim \nu |u^{n-1}-\tilde{u}^{n-1}_{H,k}|_1+ (1+\frac{\Lambda}{\lambda})(C_{\operatorname{proj}} H  + \Lambda C_{\operatorname{ol}}^{\frac{1}{2}}(\beta^{k} +  \operatorname{Tol}))\lambda^{-1} \|f\|_0.
\end{align*}
Inductively, and utilizing the geometric series,  this yields
\begin{align*}
   |\tilde{u}^n_{H,k} -u^n|_1
   & \lesssim\nu^n |u^{0}-\tilde{u}^{0}|_1+ \frac{1}{1-\nu}(1+\frac{\Lambda}{\lambda})(C_{\operatorname{proj}} H  + \Lambda C_{\operatorname{ol}}^{\frac{1}{2}}( \beta^{k} +   \operatorname{Tol}))\lambda^{-1} \|f\|_0. \qedhere
\end{align*}

\end{proof}
\end{theorem}

Up to this point, our method has been based on the Kačanov iterative method. However,    the Ka\v{ca}nov method is not the only iterative scheme that can be used for solving or model problem \eqref{WeakForm}. In the next section, we provide   some derivations for  an  iterative  LOD technique in the context of the Newton method.  
\section{Adaptive iterative LOD with the Newton method}\label{Sec:NEWTONAlgorithm}
  In this section, we want to outline how to adapt the adaptive LOD technique for solving \eqref{WeakForm} with the Newton method.  Similar to before, at each iteration, we solve an (auxiliary) linear problem, which differs from the one for the Ka\v{c}anov method. Hence, we explain how the correction operator and the error indicator for the adaptive update strategy need to be modified.

 \subsection{Correction operator with the Newton method} \label{subsec: CorOpe  Via Newton method }

In the Newton method, the linearized problem for each iteration is constructed using the Fréchet derivative.  Consider the following nonlinear operator $\mathcal F: H^1_0(\Omega)\to H^{-1}(\Omega)$, $u\mapsto \mathcal F(u)$ defined via  
\begin{equation*}\label{Newtoniterative}
    \langle \mathcal{F}(u), v\rangle := \mathcal{A}(u;u,v)-(f,v), \quad \forall  v \in H^1_0(\Omega).
\end{equation*}
The Fr\'echet derivative of $\mathcal F$ at  $\phi \in H^1_0(\Omega)$ in direction $v\in H^1_0(\Omega)$ is calculated as
   \begin{equation}\label{eq:CorrectionNEwton}
        \langle \mathcal{D}\mathcal{F}(\phi)v, \psi\rangle=(\alpha(x,\phi)\nabla v, \nabla \psi) + (v\alpha_s(x, \phi)\nabla \phi, \nabla \psi) \qquad \forall \psi\in H^1_0(\Omega).
   \end{equation}  
In one step of the Newton method, $u^{n+1}$ is computed as $u^{n+1}=u^n+\rho^n$, where $\rho^n$ solves
\[\langle \mathcal D\mathcal F(u^n)(\rho^n), v\rangle =-\langle \mathcal F(u^n), v\rangle \quad \forall v\in H^1_0(\Omega).\]
Hence, for the iterative LOD with the Newton method, following Section \ref{sec: adaptive}, the corrector problems for fixed $\phi\in H^1_0(\Omega)$ read as follows: For $v_H\in V_H$,  find $\mathcal Q_\phi v_H\in W$ as the solution of
\begin{equation}\label{eq:correcFrechet}
    \langle \mathcal D\mathcal F(\phi)(\mathcal Q_\phi v_H), w\rangle=\langle \mathcal D\mathcal F(\phi)(v_H), w\rangle \quad \forall w\in W.
\end{equation}
The localized version of the correction problem \eqref{eq:correcFrechet} as presented in Section \ref{subsec: locaiztion} reads as follows: Given $v_H\in V_H$, and any $ T \in \mathcal{T}_H $,  find $\mathcal Q_{\phi,T}^kv_H\in W(N^k(T))$ as the solution of

\begin{equation}\label{eq:correcFrechetLocal}
    \langle \mathcal D\mathcal F(\phi)(\mathcal Q^k_{\phi,T}v_H), w\rangle_{N^k(T)}=\langle \mathcal D\mathcal F(\phi)(v_H), w\rangle_T \quad \forall w\in W(N^{k}(T)).
\end{equation}
Recalling \eqref{eq:CorrectionNEwton}, we emphasize that the correction problem \eqref{eq:correcFrechet} is of convection-diffusion type so that its coercivity is not directly clear. Nevertheless,  it is  shown in \cite{khraisverfuerth2025} that, in fact, \eqref{eq:correcFrechet} is coercive over $W$ whenever $H$ is sufficiently small. As result,  problem \eqref{eq:correcFrechet}  and its localized versions \eqref{eq:correcFrechetLocal}  are indeed  well-posed.

Now, we introduce the auxiliary Newton linear multiscale problem.  Let $\phi \in H^1_0(\Omega)$, we define the following auxiliary discrete linear problem corresponding to \eqref{WeakForm}: Find the solution $u_{\phi,H} $ as
\begin{equation}\label{Newtonupadtes}
    u_{ \phi,H}:=\phi+\rho_{ \phi,H},
\end{equation}
where $\rho_{ \phi,H} \in (1-\mathcal{Q}^k_\phi)V_H$ solves the following linear problem:
\begin{equation}\label{Newtonlinearproblem}
    \langle \mathcal{D}\mathcal{F}(\phi)(\rho_{ \phi,H}),v\rangle= -\langle \mathcal{F}(\phi), v\rangle, \quad \forall v \in  (1-\mathcal{Q}^k_\phi)V_H.
\end{equation}
  
 \subsection{Error Indicator for the Newton method}

Similar to Subsection \ref{subsec:errorIndic}, we introduce an error indicator associated with the localized correction computation, as presented in \eqref{eq:correcFrechetLocal}. The definition of the error indicator is presented in the following Lemma \ref{errorIndicatorLemma}.  In particular,  the derivation of the error indicator is based on the problem formulation given in \eqref{eq:correcFrechet} and its local version \eqref{eq:correcFrechetLocal}.
\begin{lemma}\label{errorIndicatorLemma}   Assume that $\psi, \xi \in H^1_0(\Omega) $, and $\mathcal{\mathcal{Q}}_{\psi,T}^k$ are given, while   $\mathcal{\mathcal{Q}}_{\xi,T}^k$ is not available. Then, 
for all $v_H \in V_H$, the following bound holds
\[|(\mathcal{\mathcal{Q}}_{\xi,T}^k-\mathcal{\mathcal{Q}}_{\psi,T}^k)v_H|_1 \lesssim    e_{T}(\xi,\psi)|v_H|_{1,T}.\]
where
\begin{equation}\label{eq: erroindicatorNewton} 
    \begin{split}
    e^2_{T}(\xi,\psi)&:= \sum_{\substack{T' \in   N^k(T)  }}  \Bigl(\|(\alpha (\xi)-\alpha (\psi))\|^2_{\infty,  T'}\max_{w|_T: w\in V_H} \frac{\|\nabla\mathcal{Q}^k_{\psi,T}w-\chi_{T}\nabla w\|_{0,T'}^2}{\|\nabla w\|^2_{0,T}}\\
    &  \qquad\qquad\qquad+ \|\alpha_u(\xi)\nabla \xi-\alpha_u(\psi)\nabla \psi\|^2_{\infty,T'}\max_{w|_T: w\in V_H}\frac{\|(\mathcal{Q}^k_{\psi,T}-\chi_{T})w\|_{0,T'}^2}{\| w\|^2_{0,T}} \Bigr).  \\
    \end{split}
\end{equation}
Moreover, we have that
\[|(\mathcal{\mathcal{Q}}_{\xi}^{k}-{\mathcal{Q}}_{\psi}^{k})v_H|^2_1 \lesssim   C_{\operatorname{ol}}  (\max_{T \in \mathcal{T}_H} e_{T}(\xi,\psi))^2|v_H|_{1}^2.\]
Here, the constant depends on    Poincaré-Friedrich's constant and $\lambda$. 
\end{lemma}
 The  proof with the derivation of the error indicator is postponed to the appendix. Note that the term $\|\alpha_u(\xi)\nabla \xi-\alpha_u(\psi)\nabla \psi\|_{\infty,T'}$  in \eqref{eq: erroindicatorNewton} is well-defined in practice, as $\nabla \xi  \text{ and } \nabla \psi$ are, by construction of the adaptive iterative LOD, discrete (fine-scale) functions. Observe that  the error indicator definition is similar  to that introduced for the  Kačanov case,  up to an  additional term. 
 This  term is expected to increase the values of the error indicator and thereby to increase the number of updated basis functions/correctors.  From a computational perspective, the adaptive LOD based on the Newton method appears to be  computationally  more expensive than based on the Kačanov scheme, especially since the problems \eqref{eq:correcFrechetLocal} and \eqref{Newtonlinearproblem} for the Newton method are of convection-diffusion type.
 
 Summarizing, the adaptive iterative LOD algorithm with the Newton method can be implemented along the idea of Algorithm \ref{AlgorithofAdaptivity}, using, however, the local correction problem  \eqref{eq:correcFrechetLocal} in Lines \ref{Construct} and \ref{recompute}, the linear multiscale problem \eqref{Newtonlinearproblem} in Lines \ref{compute} and \ref{solve_problem}, the updates of the Newton problem  \eqref{Newtonupadtes},  as well as the error indicator definition in \eqref{eq: erroindicatorNewton} in Line \ref{compute error indicator}.   
 
 An error analysis of the adaptive iterative LOD based with the Newton method does not seem to be possible with the techniques presented in this work. As \eqref{Newtonlinearproblem} is a convection-diffusion problem, extra assumptions would be needed for the convergence and error analysis. For instance, additional assumptions on the convection term and  higher regularity of the solution are used in \cite{Abdull_Vilmar_Newton,Hybrid_High_Order_Method_Newton, reg3,  SrarPollock_Nweton}. However, higher regularity $u\in H^{1+s}(\Omega)$ is usually not possible for our general heterogeneous coefficients. 
 We therefore leave a theoretical analysis to future work and investigate this approach only numerically. In particular, we will compare the adaptive iterative LOD for Newton and Kačanov method.
\section{Numerical experiments }\label{sec: SectionExperiment}
 
\begin{figure} 
    \centering
    \includegraphics[width=0.47\textwidth]{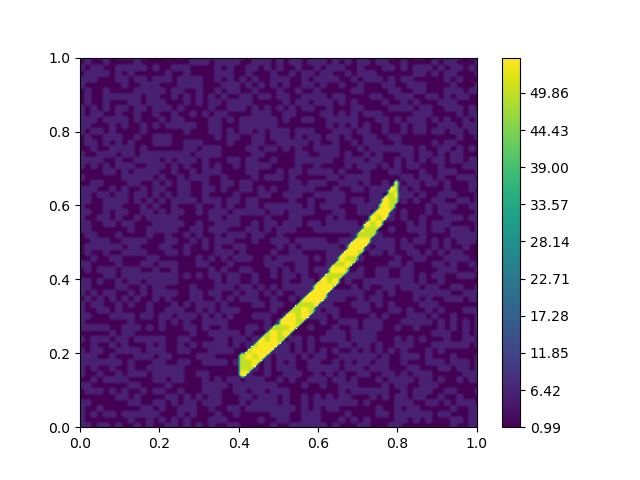}
    \caption{Spatial coefficient $c(x)$ }\label{variation}
\end{figure}

In this section, we present numerical examples to assess in detail the performance of the proposed adaptive iterative multiscale method and to validate the corresponding  theoretical findings. We consider a quasilinear problem  \eqref{WeakForm} with homogeneous Dirichlet boundary conditions and diffusion coefficients of the form 
\[\alpha(x,u)=c(x) \kappa(u), \] 
where $c$ involves fine-scale  features.    Precisely, $c$ is piecewise constant on a scale $\varepsilon=2^{-6}$ and exhibits a high-contrast channel, see Figure \ref{variation}. The considered  nonlinearities $\kappa$ are inspired by different models of the stationary Richards equation.  The domain of interest  $\Omega=[0,1]^2$ is discretized by a uniform quadrilateral mesh. We use the following coarse mesh sizes  $H=2^{-1},2^{-2},2^{-3},2^{-4},2^{-5},2^{-6}$ that do not resolve the fine-scale features of the solution $u$.  We emphasize that the analytical solution $u$ is not explicitly available. Therefore,  a standard FEM solution $u_h \in V_h$  is employed as a reference solution, where $V_h$ is the $\mathcal Q_1$ space over a fine uniform mesh of size $h=2^{-7}$, which resolves the fine-scale features of $c$.

To study  the convergence behavior of the  adaptive iterative LOD, let $\Tilde{u}^n_{H,k}$ be the multiscale solution obtained via Algorithm \ref{AlgorithofAdaptivity} at iteration $n$, and $u_h$ be the reference solution. The numerical study  presented here relies on the following relative upscaled error
\[e_{\operatorname{LOD}}:= \frac{|u_h-\Tilde{u}^n_{H,k}|_1}{|{u_h}|_1}.\]
 
Algorithm \ref{AlgorithofAdaptivity} runs till the threshold of $\operatorname{tol}=10^{-12}$ as a residual  is reached or a maximum number of 20 iterations  is attained.  The basis of the multiscale space is updated between successive iterations only if the corresponding  error indicator is larger than $\operatorname{Tol}=0.1$.  The oversampling parameter is chosen  to be fixed $k= 3$ in several experiments,  and  will clearly be mentioned, if it is changed.   The right-hand side employed in our experiment is given as 
\[
f(x) =
\begin{cases}
2^4, & \text{if } y \leq 0.35 \\
0.5,   & \text{if } \text{otherwise}. 
\end{cases}
\]
Based on Theorem \ref{th:maintheorem}, we expect $e_{\operatorname{LOD}}$ to converge linearly with respect to $H$.
Below, $\operatorname{LOD}_{\operatorname{ad}}$ refers  to the proposed adaptive iterative method in this article, and $\operatorname{LOD}_{\infty}$
refers  to the LOD method, in which one fixed multiscale space is computed once at the beginning with which then the global nonlinear problem is solved. This method has been introduced in \cite{Barbara} and formally corresponds to the adaptive iterative LOD with $\operatorname{Tol}=\infty$.   The source code for the numerical experiments is available at https://github.com/Maherkh/LodNonmonotoneNonlinearPDE.

\subsection{Van Genuchten Model}

 Here,  we apply  $\operatorname{LOD}_{\operatorname{ad}}$  to the so-called Van Genuchten model given by 
\[ {\kappa(u)}=\frac{(1+\vartheta |u|(1+(\vartheta |u|)^2)^{-\frac{1}{2}})^2}{1+(\vartheta|u|)^2},\quad \vartheta=0.005. \]
We point out that this model satisfies  Assumption \ref{assum1} imposed on the  diffusion coefficient $\alpha$.
\begin{figure} 
    \centering
    \begin{minipage}{0.45\textwidth}
       \centering      \includegraphics[width=1.1\textwidth]{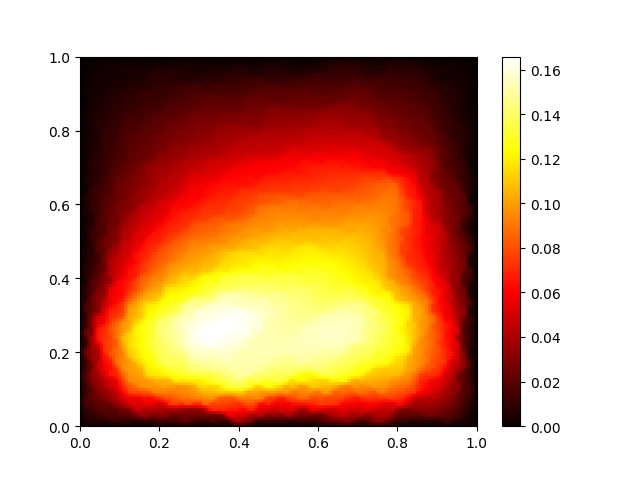} 
    
    \end{minipage}
    \hfill
    \begin{minipage}{0.45\textwidth}
         \centering  
    \includegraphics[width=1.1\textwidth]{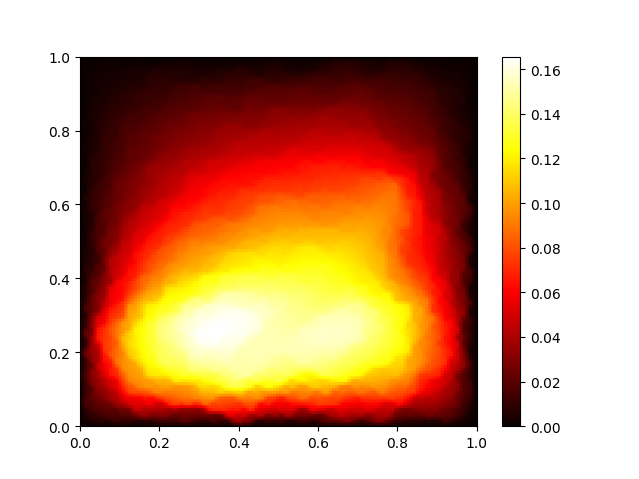}

    \end{minipage}
    \caption{Reference solution (left) and  adaptive iterative LOD solution (right).}\label{ref_lod_solution_VAN}  
\end{figure}
Figure \ref{ref_lod_solution_VAN} depicts the reference solution together with the multiscale solution obtained via the proposed $\operatorname{LOD}_{\operatorname{ad}}$ (Kačanov method) on a coarse mesh of size $H=2^{-4}$ and initial point $u^0_H=0$. We remark that the multiscale solution  captures the essential features of the reference solution despite the coarse mesh. The iterative method used in the computation terminates after 4 iterations, with a maximum of $55\%$ of the basis functions recomputed.

\begin{figure} 
    \centering
    \begin{minipage}{0.45\textwidth}
       \centering      \includegraphics[width=1.1\textwidth]{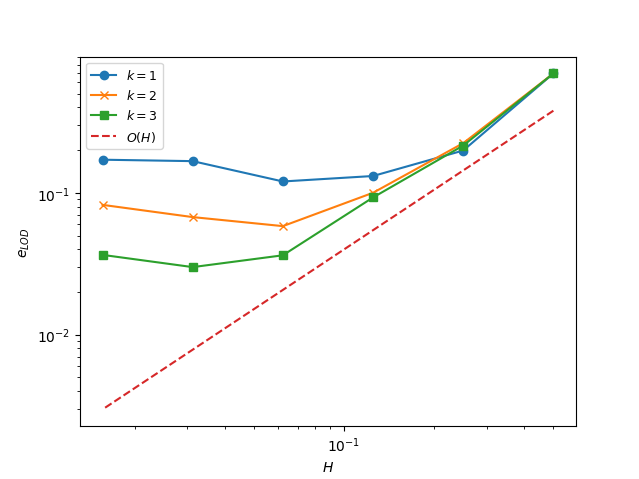}
    \end{minipage}
    \hfill
    \begin{minipage}{0.45\textwidth}
         \centering  
    \includegraphics[width=1.1\textwidth]{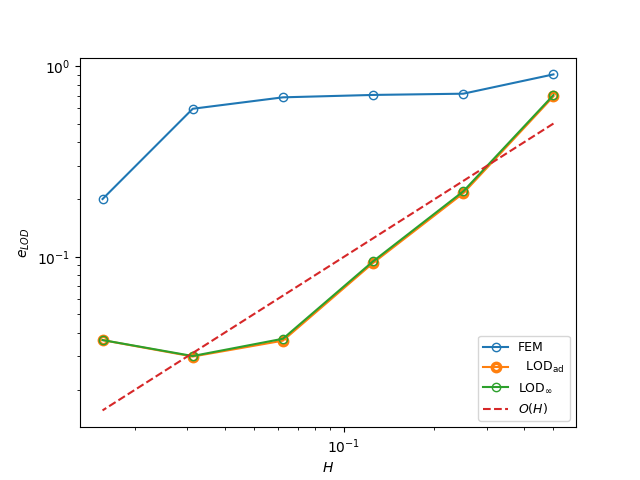}

    \end{minipage}
    \caption{Convergence history for different oversampling parameters (left) and comparison of the convergence history for different methods (right), (Van Genuchten model).}\label{localizatin and differnt method comparison}  
\end{figure}
To confirm our theoretical findings of Theorem \ref{th:maintheorem}, we consider $\operatorname{LOD}_{\operatorname{ad}}$ with initial point $u^0_H(x,y)=0.5x y (1 - x)  y  (1 - y)e^{(5(x+y))}$ for different $H$ and $k$ in Figure \ref{localizatin and differnt method comparison} (left).
 We observe the predicted linear convergence rate in the mesh size. Note that, for smaller $H$, the plateau in the  error curves reflects the dominance of  the localization error, as we   deduce from the results with increasing  the localization parameter $k$,  causing the plateau to   decrease  and the first-order convergence to continue further. 
 Moreover, we compare the proposed $\operatorname{LOD}_{\operatorname{ad}}$  with the standard FEM and $\operatorname{LOD}_{\infty}$ in Figure \ref{localizatin and differnt method comparison} (right). As expected,  the FEM error stagnates since the fine-scale features of $\alpha$ are not fully resolved. Observe that the error starts to decrease when $H$ becomes very small. In contrast,  both LOD methods show  the predicted first order convergence and substantially improved errors compared to the FEM. We emphasize that in this example,  both LOD techniques perform similarly. This behavior can be explained by the fact that    $\operatorname{LOD}_{\operatorname{ad}}$  requires a very small number of iterations (at most 5) and only few basis updates (one time  $83\%$ of the basis need to be recomputed) to approach the fixed point.

 Finally, we investigate the percentage of the basis updates given a fixed oversampling parameter. Note that all basis functions have to be computed in the zeroth iteration, so that the following numbers refer to re-computations from the first iteration onward. In general, we observe that the maximum percentage of basis updates decreases when the coarse mesh size $H$ gets smaller. However, this effect also highly depends on the initial point. For instance, for $u_H^0=0$,   the maximum  percentage of basis updates drops from $100\%$ for $H=2^{-1}$ to $15\%$ for $H=2^{-6}$ whereas for $u^0_H=  0.5x y (1 - x)  y  (1 - y)e^{(5(x+y))}$, it decreases from  $100\%$ for $H=2^{-1}$ to $83\%$ for $H=2^{-6}$. Further, the second starting point also needs slightly more iterations. 
 Hence, a ``bad'' choice of starting point may lead to more iterations and/or higher percentages for the  $\operatorname{LOD}_{\operatorname{ad}}$ as we will investigate in further detail for a more challenging model in the next experiment.

\subsection{Exponential model}

  In this section we study the exponential model given as 
\[ \kappa(u) =\exp(2u).\]
We remark that this model lies beyond the  considered setting of  Assumption \ref{assum1}. For consistency, we consider the same right-hand side as in the previous example. Unlike the previous example, we consider the $\operatorname{LOD}_{\operatorname{ad}}$ method not only using Kačanov technique in Algorithm \ref{AlgorithofAdaptivity}, but also using the Newton method introduced in Section \ref{Sec:NEWTONAlgorithm}.  

\begin{figure} 
    \centering
   \begin{subfigure}[b]{0.45\textwidth}
        \includegraphics[width=1.1\textwidth]{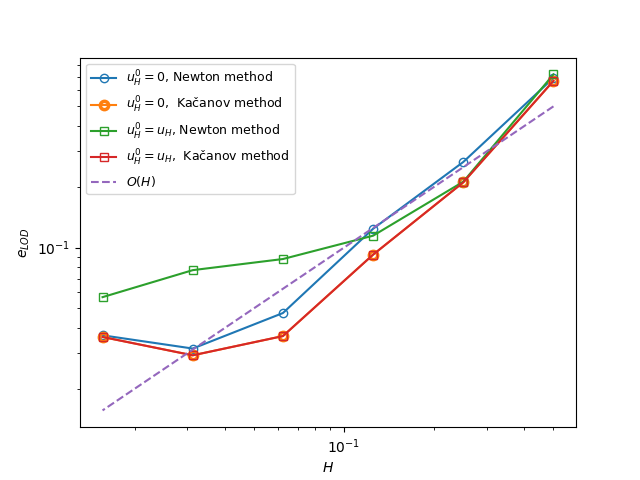}
         
    \end{subfigure}
    \hfill
     \begin{subfigure}[b]{0.45\textwidth}
        \includegraphics[width=1.1\textwidth]{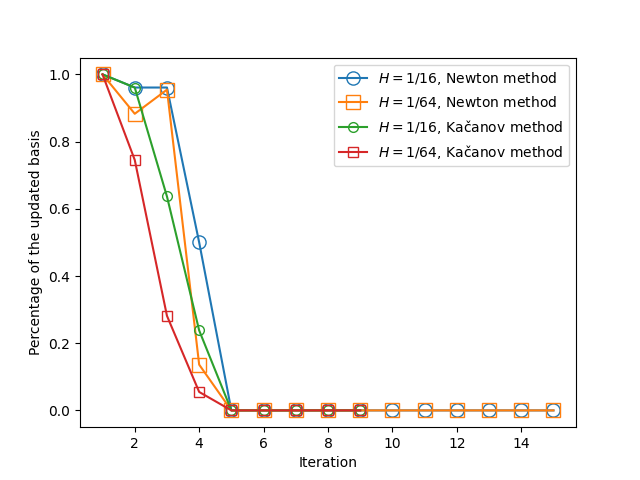}

    \end{subfigure}
    
    \caption{ Comparison of $\operatorname{LOD}_{\operatorname{ad}}$ (Kačanov) and  $\operatorname{LOD}_{\operatorname{ad}}$ (Newton)  for different initial points (left). Percentage of the updated basis in each iteration, with initial point $u^0_H=0$ (right), (exponential model).}
   \label{fig: percetage_of_the-udate}
\end{figure}

 In Figure \ref{fig: percetage_of_the-udate} (left), we examine $\operatorname{LOD}_{\operatorname{ad}}$ (Kačanov) and  $\operatorname{LOD}_{\operatorname{ad}}$ (Newton) for two different initial points, and fixed oversampling parameter.  We observe that both approaches  still exhibit first order convergence for this example. Further, the overall error behavior of both methods remains comparable, in particular for the initial point $u_H^0=0$. However, the error values for Kačanov method are less than in the case of the Newton method. In Figure \ref{fig: percetage_of_the-udate} (right), we test the performance of the methods regarding the number of basis updates for two different coarse mesh sizes and $u^0_H=0$. We observe that the percentage of the  recomputed basis decreases faster when $H$ becomes smaller for both methods.  However, for the same coarse mesh size,   $\operatorname{LOD}_{\operatorname{ad}}$(Newton) converges  more slowly  than $\operatorname{LOD}_{\operatorname{ad}}$ (Kačanov) as indicated by the higher number of iterations. Moreover, $\operatorname{LOD}_{\operatorname{ad}}$ (Newton) also needs slightly more basis updates than $\operatorname{LOD}_{\operatorname{ad}}$ (Kačanov). 
 %
 This behavior is theoretically expected, as the error indicator coincides with the error indicator of Kačanov method up to an additional term which makes it larger, and consequently causes more basis to be recomputed and  also  more iterations. In addition to the experiment in Figure \ref{fig: percetage_of_the-udate} (right), we tested both methods using another initial point $u^0_H=u_H$,  where $u_H$ is the finite element solution of \eqref{WeakForm} computed on the coarse mesh.  The number of iterations required for $\operatorname{LOD}_{\operatorname{ad}}$(Newton), for different coarse mesh sizes $H$ and for both initial points $u^0_H=0 \text{ or } u_H$, ranges between  6 and 16 iterations to converge. In contrast, the $\operatorname{LOD}_{\operatorname{ad}}$(Kačanov), requires   9 to 11 iterations. %

In Figure \ref{Fig:ourmethandAdaptiveiterati}, we compare ${\operatorname{LOD}}_{\operatorname{ad}}$ and ${\operatorname{LOD}}_\infty$ for different initial points and different iterative methods. As Figure \ref{Fig:ourmethandAdaptiveiterati} (left) illustrates, when the initial point $u_H^0(x,y)=0.5x y (1 - x)  y  (1 - y)e^{(5(x+y))}$ is chosen,   $\operatorname{LOD}_\infty$ (Kačanov) exhibits a loss of convergence; in contrast,    ${\operatorname{LOD}}_{\operatorname{ad}}$ recovers the first-order convergence behavior, as the method updates the multiscale space iteratively. In contrast, both ${\operatorname{LOD}}_\infty$ and  $\operatorname{LOD}_{\operatorname{ad}}$ (Newton) do not converge for this $u_H^0$ so that they are not shown in Figure \ref{Fig:ourmethandAdaptiveiterati} (right). The results  also indicate that the proposed ${\operatorname{LOD}}_{\operatorname{ad}}$ yields smaller errors for the Kačanov method, in particular when $u_H^0=0$. However,   the qualitative convergence behavior remains largely unaffected for both  Kačanov and Newton method when $u^0_H= 0 \text{ or } u_H$. This can likely be attributed to the proximity of the initial points to the solution $u$, which leads to a small linearization error and results in an accurate multiscale space constructed initially for ${\operatorname{LOD}}_\infty$.   
 
 \begin{figure} 
    \centering
    \begin{subfigure}[b]{0.45\textwidth}
        \includegraphics[width=1.1\textwidth]{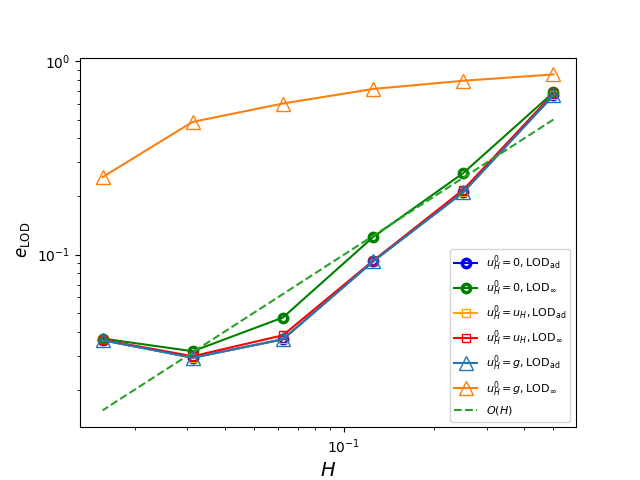}
    \end{subfigure}
    \hfill
    \begin{subfigure}[b]{0.45\textwidth}
        \includegraphics[width=1.1\textwidth]{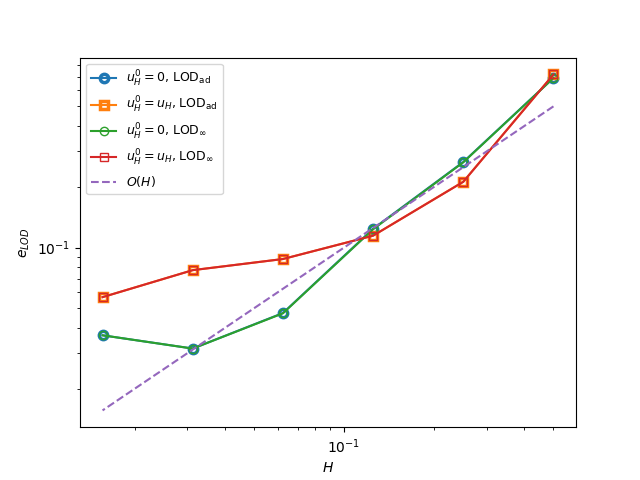}
    \end{subfigure}
    \caption{Comparison of two different LOD methods and different initial points, Kačanov method (left) and Newton method (right) (exponential model).}
    \label{Fig:ourmethandAdaptiveiterati}
    
\end{figure}

\subsection{ Examination of the smallness of the data assumption}

  In this section,  we further investigate the performance of the ${\operatorname{LOD}}_{\operatorname{ad}}$ (Kačanov) with respect to several choices of $f$  to quantify the influence of $f$ on the convergence behavior of the iterative method. The motivation for this investigation is that the right-hand side contributes to the contraction constant $\nu$ defined in \eqref{definition fo nu}  for $\operatorname{2D}$. 
  We test the method for three models: Van Genuchten Model, exponential model, and the so-called Haverkamp model, given as
  
  \[\kappa(u)=\frac{1}{1+\abs{u}}.\]
We consider a series of right-hand sides with increasing norm defined via \[
f(x) =
\begin{cases}
 
2^\gamma & \text{if } y \leq 0.15, \\
0.1 & \text{ otherwise}
\end{cases} \quad  \gamma =1, 2,4,6,8,10,12,14.
\] 
Additionally, the following parameters are fixed: the oversampling parameter $k=3$,  the coarse mesh size $ H=2^{-4}$, and   $\operatorname{Tol} =0.05.$
\begin{figure} 
    \centering
    \includegraphics[width=0.5\linewidth]{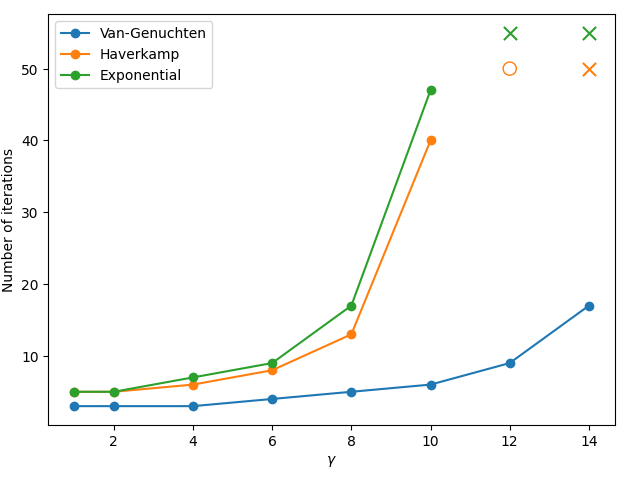}
    \caption{Increasing  numbering of iterations against  different  $f$. The notation  \textcolor{orange}{$\circ$} means the iterative method does not converge within the given maximum of iterations but might do after more iterations whereas the symbol $\times$ indicates that the iterative method is diverging. }
    \label{fig:several-right-hand-side}
\end{figure}
For the three examples, we observe an increasing  number of iterations as the values of $f$ get larger. For the Van Genuchten Model, the $\operatorname{LOD}_{\operatorname{ad}}$ (Kačanov) converges within the range of maximum 20 iterations for the different values of the right-hand side. In contrast, the exponential and the Haverkamp models have a threshold at  $\gamma=12$, after which the Haverkamp model shows only a very slow convergence (50 iterations do not suffice) or even diverging behavior. Similarly, the exponential model is not converging as $\gamma\geq 12$.  One possible explanation for this diverging behavior is that the smallness assumption $\nu <1$ is violated.  The results indicate that some condition on the size of $f$ might be required for convergence of the iteration, but the exact form for $\nu$ in \eqref{definition fo nu} is most probably not sharp.

 \section*{Conclusion}
In this work we proposed  and analyzed an adaptive  iterative numerical homogenization method for a class of nonmonotone quasilinear problems with rough coefficient. The method is theoretically and numerically investigated in the context of the Kačanov iterative scheme.  In addition, preliminary results are provided within the framework of the Newton method. In both schemes, a suitable error indicator has been derived  to adaptively update the multiscale space in each iteration. Furthermore, a first order convergence has been established for the Kačanov  method.  The error analysis also demonstrates the impact of localization, linearization,  and the adaptive update of the basis.  The introduction of additional  regularity assumption on the right-hand side 
is the key idea to  address the estimation of the nonlinearity term by  utilizing the Sobolev embedding theorem.  In 3$\operatorname{D}$, this requires an additional assumption on the contrast of the coefficient.  The  numerical experiments validate the theoretical analysis.  In addition,  we provided a numerical comparison for  both iterative schemes, in which we observed that both methods performs qualitatively the same. However, the Newton method is computationally more expensive and requires more basis updates and also more iterations. Further, we observed that the adaptive LOD  shows an improved error behavior compared to the  method previously presented  in \cite{Barbara}, in particular when  the initial point is far from the solution. Possible future research directions are the theoretical justification of the adaptive iterative LOD also for the Newton method and the extension to other types of nonlinear PDEs, for example with nonlinearity depending on the gradient or nonlinear elasticity.

\section*{Acknowledgments}
This work is funded by the Deutsche Forschungsgemeinschaft (DFG, German Research Foundation) under project number 496556642. BV also acknowledges support from the Deutsche Forschungsgemeinschaft (DFG, German Research Foundation) under Germany's Excellence Strategy – EXC-2047/2 – 390685813.


\bibliography{references}
\bibliographystyle{abbrv}  
 \appendix                  

\section{Proof of Lemma \ref{errorIndicatorLemma}}

\begin{proof}  
 Let $v_H \in V_H, \text{ and define } e:= (\mathcal{\mathcal{Q}}_{\xi,T}^k-\mathcal{\mathcal{Q}}_{\psi,T}^k)v_H$. To utilize the   ellipticity assumption in \eqref{ellip}  for problem \eqref{eq:correcFrechetLocal},  $H$ should be sufficiently small as discussed in Section \ref{subsec: CorOpe  Via Newton method }. Then,  by   definition of  Fréchet derivative \eqref{eq:CorrectionNEwton} and the local correction problem  in \eqref{eq:correcFrechetLocal}, we obtain 
 
    \begin{align*}
       \lambda |e|_1^2 &\leq \langle \mathcal D\mathcal F(\xi)(e), e\rangle_{N^k(T)}=\langle \mathcal D\mathcal F(\xi)({\mathcal {Q}}^k_{\psi,T}v_H), e\rangle_{N^k(T)}-\langle \mathcal D\mathcal F(\xi)({\mathcal {Q}}^k_{\xi,T}v_H), e\rangle_{N^k(T)}\\&=\langle \mathcal D\mathcal F(\xi)({\mathcal {Q}}^k_{\psi,T}v_H), e\rangle_{N^k(T)}-\langle \mathcal D\mathcal F(\xi)({\mathcal {Q}}^k_{\xi,T}v_H), e\rangle_{N^k(T)}-\langle \mathcal D\mathcal F(\psi)({\mathcal {Q}}^k_{\psi,T}v_H), e\rangle_{N^k(T)}\\
       &\qquad+\langle \mathcal D\mathcal F(\psi)(v_H), e\rangle_{T}\\
       &=\langle \mathcal D\mathcal F(\xi)({\mathcal {Q}}^k_{\psi,T}v_H), e\rangle_{N^k(T)}-\langle \mathcal D\mathcal F(\xi)(v_H), e\rangle_{T} -\langle \mathcal D\mathcal F(\psi)({\mathcal {Q}}^k_{\psi,T}v_H), e\rangle_{N^k(T)}\\
       &\qquad +\langle \mathcal D\mathcal F(\psi)(v_H), e\rangle_{T}\\
        &=\left((\alpha(\xi)-\alpha(\psi))(\nabla\mathcal{Q}^k_{\psi,T}v_H-\chi_{T}\nabla v_H)  + (\alpha_u(\xi)\nabla \xi-\alpha_u(\psi)\nabla \psi)(\mathcal{Q}^k_{\psi,T}v_H-\chi _Tv_H), \nabla e\right)_{N^k(T)} \\
        &\leq \Bigl(\|(\alpha (\xi)-\alpha (\psi))(\nabla\mathcal{Q}^k_{\psi,T}v_H-\chi_{T}\nabla v_H))\|_{0, N^k(T)}\\
        & \qquad+\|(\alpha_u(\xi)\nabla \xi-\alpha_u(\psi)\nabla \psi)(\mathcal{Q}^k_{\psi,T}-\chi_{T})v_H\|_{0,N^k(T)}\Bigr) \|\nabla e \|_0.
    \end{align*}
   Then it follows that   
\begin{align*}
    |e|^2 &\leq \lambda^{-2} \Bigl(\|(\alpha (\xi)-\alpha (\psi))(\nabla\mathcal{Q}^k_{\psi,T}v_H-\chi_{T}\nabla v_H)\|_{0, N^k(T)}\\
    &\qquad\qquad+\|(\alpha_u(\xi)\nabla \xi-\alpha_u(\psi)\nabla \psi)(\mathcal{Q}^k_{\psi,T}-\chi_{T})v_H\|_{0,N^k(T)}\Bigr)^2 \\
    &\leq 2\lambda^{-2} \Bigl(\|(\alpha (\xi)-\alpha (\psi))(\nabla\mathcal{Q}^k_{\psi,T}v_H-\chi_{T}\nabla v_H)\|^2_{0, N^k(T)}\\
    &\qquad\qquad+\|(\alpha_u(\xi)\nabla \xi-\alpha_u(\psi)\nabla \psi)(\mathcal{Q}^k_{\psi,T}-\chi_{T})v_H\|^2_{0,N^k(T)}\Bigr).
\end{align*}
Again, we can express the norms on $N^k(T)$ as a sum of element-wise norms of  elements $T' \in N^k(T)$. We conclude via applying H\"older’s inequality    as follows 

\begin{align*}
    |e|^2 & \lesssim    \max_{w|_T: w\in V_H} \frac{\|(\alpha (\xi)-\alpha (\psi))(\nabla\mathcal{Q}^k_{\psi,T}w-\chi_{T}\nabla w)\|_{0, N^k(T)}^2}{\|\nabla w\|^2_{0,T}} | v_H|^2_{1,T}\\
    &\qquad\qquad\qquad+\max_{w|_T: w\in V_H}\frac{\|(\alpha_u(\xi)\nabla \xi-\alpha_u(\psi)\nabla \psi)(\mathcal{Q}^k_{\psi,T}-\chi_{T})w\|_{0,N^k(T)}^2)}{\| w\|^2_{0,T}}\| v_H\|^2_{0,T}\\
    & \lesssim  \sum_{\substack{T' \in   N^k(T)  }}  \|\alpha (\xi)-\alpha (\psi)\|^2_{\infty,  T'}\max_{w|_T: w\in V_H} \frac{\|\nabla\mathcal{Q}^k_{\psi,T}w-\chi_{T}\nabla w\|_{0,T'}^2}{\|\nabla w\|^2_{0,T}} | v_H|^2_{1,T}\\
    &\qquad\qquad\qquad+\|\alpha_u(\xi)\nabla \xi-\alpha_u(\psi)\nabla \psi\|^2_{\infty,T'}\max_{w|_T: w\in V_H}\frac{\|(\mathcal{Q}^k_{\psi,T}-\chi_{T})w\|_{0,T'}^2}{\| w\|^2_{0,T}}\| v_H\|^2_{0,T}  \\
    &\lesssim   \sum_{\substack{T' \in   N^k(T)  }}  \Bigl(\|\alpha (\xi)-\alpha (\psi)\|^2_{\infty,  T'}\max_{w|_T: w\in V_H} \frac{\|\nabla\mathcal{Q}^k_{\psi,T}w-\chi_{T}\nabla w\|_{0,T'}^2}{\|\nabla w\|^2_{0,T}}\\
    &  \qquad\qquad\qquad+ \|\alpha_u(\xi)\nabla \xi-\alpha_u(\psi)\nabla \psi\|^2_{\infty,T'}\max_{w|_T: w\in V_H}\frac{\|(\mathcal{Q}^k_{\psi,T}-\chi_{T})w\|_{0,T'}^2}{\| w\|^2_{0,T}} \Bigr) \|v_H\|^2_{1,T}\\
    &\lesssim  e_{T}(\xi,\psi)^2 \|v_H\|^2_{1,T}.
\end{align*}

In the above estimate, the constant  absorbed in the notation $\lesssim$ depends only on $\lambda$. Similar to the Kačanov case, we obtain the following estimate 
\begin{align*}
    |(\mathcal{Q}^k_{\xi}-\mathcal{Q}^k_{\psi})v_H|^2 &\leq C_{\operatorname{ol}}  \sum_{T \in {\mathcal{T}}_H} |(\mathcal{Q}^k_{\xi, T}-\mathcal{Q}^k_{\psi, T})v_H|^2 \\
    &\lesssim   C_{\operatorname{ol}} \sum_{T \in {\mathcal{T}}_H}e_{T}(\xi,\psi)^2\| v_H\|_{1,T}^2. \\
    &\lesssim   C_{\operatorname{ol}} (\operatorname{max}_{T \in \mathcal{T}_H}e_{T}(\xi,\psi))^2\sum_{T \in {\mathcal{T}}_H}\| v_H\|_{1,T}^2 \\
    &\lesssim   C_{\operatorname{ol}} (\operatorname{max}_{T \in \mathcal{T}_H}e_{T}(\xi,\psi))^2 \| v_H\|_{1}^2 \\
    &\lesssim   C_{\operatorname{ol}} (\operatorname{max}_{T \in \mathcal{T}_H}e_{T}(\xi,\psi))^2|v_H|_1^2,
\end{align*}
where we used the Poincar\'e-Friedrichs inequality in the last step.
\end{proof}  
 
\end{document}